\documentclass[a4paper]{amsart}

\usepackage{amsmath,amscd,amssymb, amsthm,mathrsfs}

\newtheorem{theorem}{Theorem}[section]
\newtheorem{corollary}[theorem]{Corollary}
\newtheorem{lemma}[theorem]{Lemma}
\newtheorem{prop}[theorem]{Proposition}
\newtheorem*{theorem*}{Theorem}
\newtheorem*{hyp*}{Hyp(A)}
\newtheorem{definition}[theorem]{Definition}
\newtheorem{remark}[theorem]{Remark}

\newtheorem{example}[theorem]{Example}

\newcommand{\ind}{\mathrm{ind}}

\newcommand{\Tw}{\mathrm{Tw}}
\newcommand{\Char}{\mathrm{ch}}
\newcommand{\ord}{\mathrm{ord}} 
\newcommand{\End}{\mathrm{End}} 
\newcommand{\Ind}{\mathrm{Ind}}

\newcommand{\N}{\mathrm{N}}
\newcommand{\et}{\mathrm{\acute{e}t}}
\newcommand{\cores}{\mathrm{cores}}
\newcommand{\loc}{\mathrm{loc}}
\newcommand{\Q}{\mathbb{Q}}
\newcommand{\C}{\mathbb{C}}
\newcommand{\Cp}{\mathbb{C}_{p}}
\newcommand{\Qp}{\Q_p}

\newcommand{\Zp}{\mathbb{Z}_p}

\newcommand{\Z}{\mathbb{Z}}
\newcommand{\Frob}{\mathrm{Frob}}
\newcommand{\Aut}{\mathrm{Aut}}
\newcommand{\Sel}{\mathrm{Sel}}
\newcommand{\Gal}{\mathrm{Gal}}
\newcommand{\cyc}{\mathrm{cyc}}
\newcommand{\Hom}{\mathrm{Hom}}

\newcommand{\Spec}{\mathrm{Spec}}

\newcommand{\pr}{\mathrm{pr}}
\newcommand{\Hyp}{\mathrm{Hyp}}
\newcommand{\CH}{\mathrm{CH}}

\newcommand{\D}{\mathrm{D}}
\newcommand{\ur}{\mathrm{ur}}
\newcommand{\coker}{\mathrm{coker}}

\title{An Euler system for characters over an imaginary biquadratic field}

\author[J.~Lamplugh]{Jack Lamplugh}
\address[Lamplugh]{Department of Mathematics \\
University College London\\
Gower Street, London WC1E 6BT, UK}
\email{j.lamplugh@ucl.ac.uk}

\date{\today}

\begin{document}

\begin{abstract}
Given a pair of modular forms with complex multiplication by distinct imaginary quadratic fields, the four dimensional Galois representation associated to their Rankin--Selberg convolution is induced from a character over an imaginary biquadratic field $F$. Using the Euler system of Lei, Loeffler and Zerbes we bound a Selmer group associated to this character, over the unique $\Z_{p}^{3}$-extension of $F$.
\end{abstract}

\maketitle

\section{Introduction}
We fix an algebraic closure $\overline{\Q}$ of $\Q$ and a prime $p\geq 5$. We then fix embeddings $\iota_{\infty} : \overline{\Q} \rightarrow \C$ and $\iota_{p}: \overline{\Q} \rightarrow \Cp$, and let $L$ denote a finite extension of $\Qp$ with ring of integers $\mathcal{O}$. Let $F$ be an imaginary biquadratic field for which $p$ splits completely. By class field theory, there exists a unique $\Z_{p}^{3}$-extension of $F$, which we denote by $F_{\infty}$. We further suppose that $\theta : G_{F}:= \Gal(\overline{\Q}/F) \rightarrow \mathcal{O}^{\times}$ is a finite order character, with order coprime to $p$, and we put $F'= \overline{\Q}^{\ker\theta}$ and $\Delta = \Gal(F'/F)$. We also fix a $p$-adic CM-type type $\Sigma$ for $F$, that is to say a subset of the primes of $F$ above $p$ such that $\Sigma$ and $\Sigma':=c(\Sigma)$ form a disjoint union of the primes of $F$ above $p$, where $c$ denotes complex conjugation. Such a choice distinguishes an imaginary quadratic subfield $K$ of $F$ determined by the property that $\Sigma$ is fixed setwise by $\Gal(F/K)$. We denote the other imaginary quadratic subfield of $F$ by $K'$. 

We put $F'_{\infty} = F' F_{\infty}$, and let $G=\Gal(F'_{\infty}/F')\cong \Z_{p}^{3}$. We let $M^{\Sigma'}_{\Sigma}(F'_{\infty})$ denote the maximal abelian $p$-extension of $F'_{\infty}$ which is unramified outside the primes above $\Sigma'$, and for which the primes above $\Sigma$ split completely. We then write $X^{\Sigma'}_{\Sigma}(F'_{\infty})$ for $\Gal(M^{\Sigma'}_{\Sigma}(F'_{\infty}) / F'_{\infty})$, which has the natural structure of a $\Lambda(G)=\Zp[[G]]$-module, as well as an action of $\Delta$. Since $p$ is coprime to $[F':F]$, we have the isotypical decomposition $X^{\Sigma'}_{\Sigma}(F'_{\infty})\otimes \mathcal{O} = \bigoplus_{\chi} X^{\Sigma'}_{\Sigma}(F'_{\infty}) ^{\chi},$ where $\chi$ runs over the characters of $\Delta$. Our aim in this paper is to show that $X^{\Sigma'}_{\Sigma}(F'_{\infty}) ^{\theta}$ is a finitely generated torsion $\Lambda(G)$-module and to show its characteristic ideal is bounded  by an appropriate $p$-adic $L$-function.

We are successful only under some assumptions on $\theta$. Before we state these hypotheses, we introduce some notation. For a character $\chi$ of $G_{F}$ and an element $g\in \Gal(F/\Q)$, $\chi^{g}$ denotes the character $\chi^{g}(\tau)=\chi(\tilde{g}^{-1}\tau \tilde{g})$, where $\tilde{g}\in \Gal(\overline{\Q}/\Q)$ denotes any lift of $g$. We let $\sigma,\sigma' \in \Gal(F/\Q)$, be generators of $\Gal(F/K)$ and $\Gal(F/K')$ respectively. Let $D_{p}$ be the decomposition group in $G_{F}$ corresponding to $\iota_{p}$. Finally we let $\omega$ denote the mod $p$ cyclotomic character $\omega : G_{F} \rightarrow \Aut(\mu_{p})\subset \Z_{p}^{\times}$.

\begin{hyp*} A finite order character $\vartheta$ of $G_{F}$ satisfies hypothesis A if there exists an integer $i$ and finite order characters $\chi$ and $\chi'$ of $G_{K}$ and $G_{K'}$ respectively, which are unramified above $p$ such that (i) $\rho:=\vartheta\omega^{-i}=\chi\chi'$, (ii)  $  \ker(\rho^{\sigma}) \neq \ker(\rho)$, (iii) $\vartheta^{\sigma}_{\mid D_{p}}\neq \vartheta_{\mid D_{p}}$ and (iv) $\vartheta^{\sigma'}_{\mid D_{p}}\neq \vartheta_{\mid D_{p}}$.
\end{hyp*}

 The embedding $\iota_{p}$ determines a prime $\mathfrak{P}$ of $F$ above $p$, and we assume that $\mathfrak{P}\in \Sigma$. We write $\Gamma = \Gal(K_{\infty}/K)$, where $K_{\infty}/K$ is the unique $\Zp$-extension of $K$ unramified outside the prime below $\mathfrak{P}$. We define $\Gamma'$ similarly, with $K$ replaced by $K'$, and we also write $\Gamma_{\cyc}=\Gal(\Q_{\infty}/\Q)$, where $\Q_{\infty}$ denotes the cyclotomic $\Zp$-extension of $\Q$. As explained below the natural map $G \rightarrow \Gamma \times \Gamma' \times \Gamma_{\cyc}$ is an isomorphism.
We assume that  that $\vartheta:= \theta^{-1}\omega$ satisfies $Hyp(A)$, and fix  $\chi,\chi'$ and $i$ modulo $p-1$ as above.  Then we let $I(\chi) \subset  \Lambda_{\mathcal{O}}(\Gamma)$ denote Hida's congruence ideal associated to the Hida family $\sum_{\mathfrak{a}} \chi(\mathfrak{a}) \sigma_{\mathfrak{a}}^{-1}q^{\N_{K/\Q}(\mathfrak{a})}\in \Lambda(\Gamma)[[q]]$. Assuming $\mathcal{O}$ is sufficiently large, there exists a $p$-adic $L$-function $L_{p,i}^{\chi,\chi'}\in I(\chi)^{-1}\Lambda_{\mathcal{O}}(G)$, interpolating the complex $L$-values $L(\vartheta\rho,0)$, as $\rho$ runs over a certain subset of algebraic Hecke characters of $F$ with conductor supported on the primes above $p$. 

\begin{theorem}\label{thm1} Suppose that $\vartheta := \theta^{-1}\omega$ satisfies $Hyp(A)$, and assume that $\chi,\chi'$ and $i$ are chosen as above. Then 
$
X^{\Sigma'}_{\Sigma}(F'_{\infty}) ^{\theta}
$
is a finitely generated torsion $\Lambda_{\mathcal{O}}(G)$-module and 
$$
\Char X^{\Sigma'}_{\Sigma}(F'_{\infty}) ^{\theta} \text{ divides } I(\chi)L_{p,i}^{\chi,\chi'} .
$$
\end{theorem}

We remark that our results allow us bound a Selmer group associated to the Galois representation $\mathcal{O}(\theta)$. When the order of $\theta$ has order coprime to $p$, this Selmer group is naturally isomorphic to $X^{\Sigma'}_{\Sigma}(F'_{\infty})^{\theta}$. Our method also extends to the case when $p$ divides the order of $\theta$, but the result is stated as above for ease of exposition.

We now give a brief outline of the strategy of the proof. Given a pair of modular forms there exists a four dimensional Galois representation of $G_{\Q}$ associated to their Rankin--Selberg convolution. In \cite{LLZ1} Lei, Loeffler and Zerbes have constructed an Euler system for this representation with classes over abelian extensions of $\Q$. If at least one of the forms has $CM$  by an imaginary quadratic field $K$, then this Galois representation is induced from a two-dimensional representation over $G_{K}$.  In this case Lei, Loeffler and Zerbes have also constructed an Euler system, with classes defined over abelian extensions of $K$, for this two-dimensional representation \cite{LLZ2}.    
Our assumption on $\vartheta$ implies that the Galois representation $\mathcal{O}(\vartheta)$, when induced to a $G_{\Q}$-representation, can be interpreted as the Galois representation associated to the Rankin--Selberg convolution of two CM modular forms. Following the method of \cite{LLZ2}, we can construct Euler system classes for $\mathcal{O}(\vartheta)$ over abelian extensions of $F$ which are obtained from composites of abelian extensions of $K$ and $K'$. This method does not provide us with enough Euler system classes to apply Rubin's Euler system machinery and bound the associated Selmer group directly, but crucially we can construct classes over the $\Z_{p}^{3}$-extension of $F$. The strategy is then to twist the Euler system by arbitrary continuous characters $\rho$ of $G$ and apply Shapiro's lemma to obtain an Euler system over $K$ for $\Ind^{K}_{F}\mathcal{O}(\vartheta\rho)$ for each character $\rho$. By applying the results of Rubin to each of these Euler systems, we obtain the above result. A crucial ingredient of the proof is the explicit reciprocity law \cite[Theorem B]{KLZ2} which relates the Euler system to a three variable $p$-adic $L$-function, and in particular proves the Euler system is not trivial. 

We make some brief remarks about the conditions in $Hyp(A)$.  We note that by a twisting argument we can relax the hypothesis slightly by only asking that $Hyp(A)$ is satisfied by some twist of $\vartheta$ by a finite order character of $G$.  The first condition is necessary to relate the Galois representation to one arising from a Rankin--Selberg convolution. The second condition relates to the hypotheses $Hyp(K_{\infty},T)$ of \cite[II.3]{R}. The final two conditions are required in the construction of the Euler system, see Theorem \ref{ind}.  

 \subsection*{Acknowledgement}
 The author would like to thank David Loeffler and Sarah Zerbes for many insightful conversations. The author was supported by the Heilbronn Institute for Mathematical Research.

\section{Notation and Preliminaries}
We briefly introduce the notation and conventions used throughout this paper. 
\subsection{Euler systems and twisting}\label{esandtw}
Let $p\geq 5$ be prime and $L/\Qp$ a finite extension with ring of integers $\mathcal{O}$. Suppose that $K$ is a number field and $T$ is a finite rank free $\mathcal{O}$-module with a continuous action of $G_{K}=\Gal(\overline{\Q}/K)$ unramified outside a finite set of places. We write $T^{*}$ for the dual representation $T^{*}= \Hom_{\mathcal{O}}(T,\mathcal{O})$ and for $n\in \Z$, write $T(n)$ for the representation $T$ twisted by the $n$-th power of the cyclotomic character. We let $S$ denote a finite set of places of $K$ containing all the infinite places, all the places dividing $p$ and all the places at which $T$ is ramified. We write $K^{S}$ for the maximal algebraic extension of $K$ which is unramified outside $S$, and for $F/K$ contained in $K^S$, we write $H^{1}_{S}(F,T)$ for $H^{1}(\Gal(K^{S}/F),T)$.

For an integral ideal $\mathfrak{r}$ of $K$, we let $K(\mathfrak{r})$ denote the maximal abelian $p$-extension of $K$ with conductor dividing $\mathfrak{r}$. We suppose that $K_{\infty}/K$ is a $\Z_{p}^{d}$-extension for some integer $d\geq 1$, and we let $K_{\infty}(\mathfrak{r})$ denote the composite $K_{\infty}K(\mathfrak{r})$ and put $G(\mathfrak{r})=\Gal(K_{\infty}(\mathfrak{r})/K)$. We write $\Lambda_{\mathcal{O}}(G(\mathfrak{r}))=\mathcal{O}[[G(\mathfrak{r})]]$ for the completed group ring of $G(\mathfrak{r})$, and $\Lambda^{\#}_{\mathcal{O}}(G(\mathfrak{r}))$  for the free rank one $\Lambda_{\mathcal{O}}(G(\mathfrak{r}))$-module endowed with the action of $G_{K}$ by the inverse of the canonical character. We also let $S(\mathfrak{r})$ denote a finite set of primes containing $S$ and the primes dividing $\mathfrak{r}$. Then by Shapiro's lemma we have a canonical isomorphism of $\Lambda_{\mathcal{O}}(G(\mathfrak{r}))$-modules
\begin{equation}\label{ShapiroIwCoh}
 H^{1}_{S(\mathfrak{r})}(K,\Lambda_{\mathcal{O}}^{\#}(G(\mathfrak{r}))\otimes_{\mathcal{O}} T ) \cong \varprojlim_{F} H_{S(\mathfrak{r})}^{1}(F,T),
\end{equation}
where $F$ runs over the finite extensions of $K$ contained in $K_{\infty}(\mathfrak{r})$, and the transition maps are given by corestriction.

An Euler system for $(T,K_{\infty})$ consists of cohomology classes 
$$
c_{\infty}(\mathfrak{r}) \in H^{1}_{S(\mathfrak{r})}(K,\Lambda_{\mathcal{O}}^{\#}(G(\mathfrak{r}))\otimes_{\Zp} T ),
$$
for all integral ideals $\mathfrak{r}$ coprime to $S$, satisfying the relations
$$
\pi^{\mathfrak{r}'}_{\mathfrak{r}} (c_{\infty}(\mathfrak{r'})) = \prod_{\mathfrak{l}} P_{\mathfrak{l}} (\Frob_{\mathfrak{l}}^{-1})\cdot c_{\infty}(\mathfrak{r'})
$$
where the product runs over all primes $\mathfrak{l}$ ramified in $K(\mathfrak{r}')$ but not in $K(\mathfrak{r})$.
Here $\pi^{\mathfrak{r}'}_{\mathfrak{r}}$ denotes the restriction map $\Lambda_{\mathcal{O}}(G(\mathfrak{r'}))\rightarrow \Lambda_{\mathcal{O}}(G(\mathfrak{r}))$ and
$$
P_{\mathfrak{l}}(X) = \det\left(1-X\cdot\Frob_{\mathfrak{l}}^{-1} \mid T^{*}(1)\right)\in \mathcal{O}[X].
$$
 Here and throughout the paper $\Frob_{\mathfrak{l}}$ denotes the arithmetic Frobenius at $\mathfrak{l}$. We note that to apply the results of \cite[II.3]{R}, we can restrict to ideals $\mathfrak{r}$ which are products of split primes of squarefree norm.

We let $G=\Gal(K_{\infty}/K)$, and for a continuous character $\rho : G \rightarrow \mathcal{O}^{\times}$  and $\mathfrak{r}$ as above we define
$$
\Tw_{\rho}: \Lambda_{\mathcal{O}}(G(\mathfrak{r})) \rightarrow  \Lambda_{\mathcal{O}}(G(\mathfrak{r}))
$$
to be the continuous $\mathcal{O}$-linear map which extends $g\mapsto \rho(g)g$ for all $g\in G(\mathfrak{r})$. Then we have a canonical isomorphism of $\Lambda_{\mathcal{O}}(G(\mathfrak{r}))$-modules 
\begin{equation*}
\alpha_{\rho} : H^{1}_{S(\mathfrak{r})}(F, \Lambda_{\mathcal{O}}^{\#}(G(\mathfrak{r}))\otimes_{\mathcal{O}}T ) \otimes_{\mathcal{O}}\mathcal{O}(\rho) \cong  H^{1}_{S(\mathfrak{r})}(F, \Lambda_{\mathcal{O}}^{\#}(G(\mathfrak{r}))(\rho) \otimes T)
\end{equation*}
described on cocycles by the formula
$$
\alpha_{\rho}(z\otimes w)(g) = \Tw_{\rho^{-1}}(z(g))\otimes w. 
$$ 
Note that if $\lambda\in \Lambda_{\mathcal{O}}(G(\mathfrak{r}))$, then we have 
\begin{equation}\label{twista}
\alpha_{\rho}(\lambda z\otimes w) =  \Tw_{\rho^{-1}}(\lambda)\alpha_{\rho}(z\otimes w).
\end{equation}
Given a cohomology class  $c\in H^{1}_{S(\mathfrak{r})}(F, \Lambda_{\mathcal{O}}^{\#}(G(\mathfrak{r}))\otimes_{\mathcal{O}}T )$, we write $c^{(\rho)}$ for the class defined by $\alpha_{\rho}(c\otimes 1)$. 
\subsection{Algebraic Hecke characters}
Suppose that $K$ is a totally imaginary number field and $\mathfrak{f}$ is an integral ideal of $K$. Let $\mathrm{Emb}(K,\overline{\Q})$ denote the set of field embeddings of $K$ into $\overline{\Q}$ and let $I(\mathfrak{f})$ denote the group of fractional ideals coprime to $\mathfrak{f}$. Given an element $\omega \in \mathbb{Z}[\mathrm{Emb}(K,\overline{\Q})]$, we say that a homomorphism $\psi : I(\mathfrak{f}) \rightarrow \overline{\Q}^{\times}$ is an algebraic Hecke character of $K$ of modulus $\mathfrak{f}$ and infinity type $\omega$ if for all $\alpha \in K^{\times}$ with $\alpha \equiv^{\times} 1$ modulo $\mathfrak{f}$, we have 
$$
\psi((\alpha)) = \alpha ^ \omega.
$$
The largest such ideal with this property is called the conductor of $\psi$. When $K$ is an imaginary quadratic field and $\mathfrak{a}$ an ideal of $K$, we let $\overline{\mathfrak{a}}$ denote the ideal conjugate to $\mathfrak{a}$ under the non-trivial element of $\Gal(K/\Q)$. We use the notation $\psi^{c}$ to denote the algebraic Hecke character of $K$ defined by $\psi^{c}(\mathfrak{a}):=\psi(\overline{\mathfrak{a}})$ for all $\mathfrak{a}\in I(\overline{\mathfrak{f}})$.

Given fixed embeddings $\iota_{\infty}: \overline{\Q} \rightarrow \C$ and $\iota_{p}: \overline{\Q} \rightarrow \C_{p}$, we define $\psi_{\infty}:= \iota_{\infty} \circ \psi$ and  $\psi_{p}:= \iota_{p} \circ \psi$. We also consider $\psi_{p}$ as a Galois character $\psi_{p}:G_{K}\rightarrow \overline{\Q}_{p}^{\times}$, which is the unique continuous character such that $\psi_{p}(\Frob_{\mathfrak{l}})= \psi_{p}(\mathfrak{l})$, for all primes $\mathfrak{l}$ of $K$ coprime to $\mathfrak{f}$.

\subsection{Modular curves and cohomology}
For an integer $N\geq 5$, $Y_{1}(N)$ denotes the $\Q$-scheme whose points over a $\Q$-scheme $S$ correspond to isomorphism classes $(E,P)$, where $E/S$ is an elliptic curve, and $P\in E(S)$ has order $N$. For integers $N\geq 5$ and $d\geq 1$ we define the morphisms 
$\pr_{1}:Y_{1}(Nd)\rightarrow Y_{1}(N)$ by $(E,P)\mapsto  (E,[d]P)$ and   $\pr_{2}:Y_{1}(Nd)\rightarrow Y_{1}(N)$ by $(E,P)\mapsto  (E/<NP>, P+<NP>)$.

The dual Hecke operators $T'(n)$ for $n\geq 1$ on $Y_{1}(N)$ are as defined in \cite[Section 2.9]{Kato}, and we denote by $\mathbb{T}(N)$ the $\Zp$-subalgebra of $\End(H^{1}_{\et}(\overline{Y_{1}(N)},\Zp))$ generated by $T'(n)$ for all $n\geq 1$. If $\ell$ is a prime dividing $N$, then we write $U'(\ell)$ for $T'(\ell)$. When $p$ divides $N$, we let $e'_{\ord}:=\varinjlim_{n} U'(p)^{n!}$ denote Hida's anti-ordinary projector, and we put $\mathbb{T}_{\ord}(N):= e'_{\ord} \mathbb{T}(N)$. 

We define
$$
GES_{p}(N,\mathbb{Z}_{p}) := \varprojlim_{r\geq 0} H^{1}_{\et}(\overline{Y_{1}(Np^{r})},\Zp),
$$
with the limit being taken with respect to $\pr_{1,*}$. As a $G_{\Q}$-module, it is unramified outside the primes dividing $Np$. The actions of the adjoint Hecke operators at each level are compatible and thus we can define an action of the adjoint Hecke operators on $GES_{p}(N,\mathbb{Z}_{p})$. We define
\begin{align*}
H^{1}_{\ord}(Np^{\infty}) &= e'_{\ord} GES_{p}(N,\mathbb{Z}_{p})(1)\\
&\cong \varprojlim_{r\geq 1} e'_{\ord}  H^{1}_{\et}(\overline{Y_{1}(Np^{r})},\Zp(1)). 
\end{align*}

Let $Z_{N}=\varprojlim_{r\geq 1} \left(\Z/Np^{r}\Z \right)^{\times}$ and $\Zp[[Z_{N}]]:= \varprojlim_{r\geq 1} \Zp[\left(\Z/Np^{r}\right)^{\times}]$. Then we have an action of $\Zp[[Z_{N}]]$ on $H^{1}_{\ord}(Np^{\infty})$ with $u\in Z_{N}$ acting via $\left<u^{-1} \text{ mod }Np^{r}\right>$ at level $Np^{r}$.  
We let $\mathbb{T}(Np^{\infty}) \subset \End(H^{1}_{\ord}(Np^{\infty}))$ denote the algebra generated over $\Zp[[Z_{N}]]$ by $T'(n)$ for all $n\geq 1$. There is a control theorem due to Ohta \cite{Ohta99,Ohta00} which allows us to describe the $\mathbb{T}(Np^{\infty})[G_{\Q}]$-structure of $H^{1}_{\ord}(Np^{\infty})$ in terms of the $\mathbb{T}_{\ord}(Np^{r})[G_{\Q}]$-structure of $e'_{\ord}  H^{1}_{\et}(\overline{Y_{1}(Np^{r})},\Zp(1))$ for all $r\geq 1$. We refer to \cite[Proposition 7.2.1]{KLZ2} for a statement of this theorem. 

Let $f=\sum_{n\geq 1} a_{n}(f)q^{n}\in S_{2}(\Gamma_{1}(N),\C)$ be a normalised weight two cuspform which is an eigenform for all the Hecke operators $T(\ell)$ and $U(\ell)$ with $\ell$ prime and $\ell\nmid N$ and $\ell \mid N$ respectively. If $L$ is a a number field containing the Fourier coefficients of $f$, and $\mathfrak{P}$ is a prime of $L$ lying above $p$, then we define $V_{L_\mathfrak{P}}(f)^{*}$ to be the maximal quotient of $ H^{1}_{\et}(\overline{Y_{1}(N)},L_{\mathfrak{P}}(1))$ on which $T'(\ell)$  (or $U'(\ell)$) act as multiplication by $a_{\ell}(f)$. It is an irreducible two-dimensional representation such that for all primes $\ell \nmid Np$ it is unramified and 
\begin{equation}\label{charpolV}
\mathrm{det}_{L_{\mathfrak{P}}} ( 1-X\cdot \Frob_{\ell} \mid V_{L_\mathfrak{P}}(f)^{*}) = 1-a_{\ell}(f)X + \varepsilon(\ell) \ell X^{2},
\end{equation}
where $\varepsilon : \left(\Z/N\Z\right)^{\times} \rightarrow \C^{\times}$ is the character of $f$.

\section{Construction of the Euler system}\label{seccon}
We now give an outline of the construction of an Euler system associated to a pair of modular forms with CM by distinct imaginary quadratic fields. 
\subsubsection{Asymmetric zeta elements}

To construct the Euler system classes we begin with the asymmetric zeta elements introduced by Lei--Loeffler--Zerbes \cite{LLZ1,LLZ2}. We do not recall the definition here, but instead list below some basic properties we will need.  
For integers $m\geq 1$ and $N',N\geq 5$ the motivic cohomology classes 
$$
{}_{c}\Xi(m,N,N') \in CH^{2}(Y_{1}(N)\times Y_{1}(N')\times \Spec(\mathbb{Q}(\mu_{m})),1)
$$
are those denoted by ${}_{c}\Xi(m,N,N',1)$ \textit{op. cit.} 
The following describes how the asymmetric zeta classes behave with respect to the degeneracy maps $\pr_{1}$ and $\pr_{2}$. For a prime $\ell$, we have the following relations, which can be found in \cite[Section 2.2]{LLZ2},
\begin{equation}\label{1stdeg}
(\pr_{1}\times 1)_{*}({}_{c}\Xi(m,\ell N,N')) = \begin{cases}  {}_{c}\Xi(m,N,N') &\text { if }\ell \mid mNN' \\
\left( 1-(\left<\ell^{-1}\right>,\left<\ell^{-1}\right>)\sigma_{\ell}^{-2} \right){}_{c}\Xi(m,N, N') &\text{ else.}
\end{cases}
\end{equation}
If $\ell\nmid mNN'$, then 
\begin{equation}\label{2nddeg}
(\pr_{2}\times 1)_{*} {}_{c}\Xi(m,\ell N, N') = \left[ (1,T'(\ell))\sigma_{\ell}^{-1} - (T'(\ell),\left<\ell^{-1} \right>)\sigma_{\ell}^{-2}   \right]   {}_{c}\Xi(m,N,N').
\end{equation}
We also recall the dependence of the zeta elements on the integer $c$, which can be found in \cite[Proposition 2.7.5]{LLZ1}. For integers $c,d$ coprime to $6mNN'p$ we have 
\begin{equation}\label{cdsymm}
\left[ d^2 - \left(\left< d\right>,\left< d\right> \right) \sigma_{d}^{2}  \right] {}_{c}\Xi(m,N,N') = \left[ c^2 - \left(\left< c\right>,\left< c\right> \right) \sigma_{c}^{2}  \right]{}_{d}\Xi(m,N,N').
\end{equation}
Here $\sigma_{c}$ denotes the element of $\Gal(\Q(\mu_{m})/\Q)$ given by $\zeta\mapsto \zeta^{c}$ for $\zeta\in \mu_{m}$, and similarly for $\sigma_{d}$.

\subsubsection{Beilinson--Flach classes}
By applying the \'{e}tale regulator map 
$$
r_{\et} :  \CH^{2}(Y_{1}(N)\times Y_{1}(N')\times \Spec(\mathbb{Q}(\mu_{m})),1) \rightarrow H^{1}(\Q(\mu_{m}), H^{2}_{\et}(\overline{Y_{1}(N)\times Y_{1}(N')}, \Zp(2))
$$
and the K\"{u}nneth formula 
$$
H^{2}_{\et}(\overline{Y_{1}(N)\times Y_{1}(N')},\Zp)\cong   H^{1}_{\et}(\overline{Y_{1}(N)},\Zp) \otimes  H^{1}_{\et}(\overline{Y_{1}(N')},\Zp)
$$
one obtains the classes 
$$
{}_{c}z(m,N,N') \in  H^{1}(\Q(\mu_{m}), H^{1}_{\et}(\overline{Y_{1}(N)},\Zp(1))\otimes H^{1}_{\et}(\overline{Y_{1}(N')},\Zp(1)).
$$
In the following we fix $S$ to be a finite set of primes, containing those dividing $mNN'p$. The classes ${}_{c}z(m,N,N')$ are then unramified outside $S$ \cite[Proposition 6.5.4]{LLZ1}.  
By the relation \eqref{1stdeg} and the compatibility of the \'{e}tale regulator map with pushforward maps, the elements ${}_{c}z(m,Np^{r},N'p^{r})$ are compatible for $r\geq 1$ under the pushforward $(\pr_{1}\times \pr_{1})_{*}$. We now define the asymmetric Rankin--Iwasawa classes
$$
{}_{c}\mathcal{RI}(m,N,N') \in H^{1}_{S}(\Q(\mu_{m}),GES_{p}(N,\Zp)(1) \widehat{\otimes} GES_{p}(N',\Zp)(1)),
$$
such that 
$$
{}_{c}\mathcal{RI}(m,N,N')  = \varprojlim_{r\geq 1}{}_{c}z(m,Np^{r} , N'p^{r}). 
$$

\begin{theorem}
Suppose that $\ell$ is a prime such that $\ell \nmid mNN'pc$. Then we have 
$$
(\pr_{1}\otimes 1)_{*}{}_{c}\mathcal{RI}(m, \ell N,N') = [1-(\left<\ell^{-1}\right> \otimes \left<\ell^{-1}\right>)\sigma_{\ell}^{-2}] {}_{c}\mathcal{RI}(m,N,N')
$$
and 
$$
( \pr_{2}\otimes 1)_{*}{}_{c}\mathcal{RI}(m,\ell N, N' ) = [(1\otimes T'(\ell))\sigma_{\ell}^{-1} - (T'(\ell)\otimes\left<\ell^{-1}\right>)\sigma_{\ell}^{-2})] {}_{c}\mathcal{RI}(m,N,N')
$$
\end{theorem}
\begin{proof} This follows from \cite[Theorem 2.2.2]{LLZ2} and the compatibility of adjoint Hecke operators under the control theorem, and the compatibility of the \'{e}tale regulator map with correspondences. 
\end{proof}

By \cite[Theorem 3.3.2]{LLZ1} we have the norm relation
$$
\cores_{\Q(\mu_{mp^{r}})}^{\Q(\mu_{mp^{r+1}})} {}_{c}\mathcal{RI}(mp^{r+1},N,N') =
\begin{cases}
(U'(p)\times U'(p) - \sigma_{p})  {}_{c}\mathcal{RI}(m,N,N')  &r=0\\
(U'(p)\times U'(p)) {}_{c}\mathcal{RI}(mp^{r},N,N') & r\geq 1
\end{cases}
$$

We may now define the asymmetric Beilinson--Flach classes. Recall that $H^{1}_{\ord}(Np^{\infty})$ is the maximal quotient of $GES_{p}(N,\Zp)(1)$ for which the Hecke operator $U'(p)$ is invertible. Thus we define, for $p\nmid mNN'$,  
$$
{}_{c}\mathcal{BF}(m,N,N') \in H^{1}_{S}(\Q,H^{1}_{\ord}(Np^{\infty}) \widehat{\otimes} H^{1}_{\ord}(N'p^{\infty}) \widehat{\otimes} \Lambda^{\#}(\Gal(\Q(\mu_{mp^{\infty}})/\Q)))
$$
as the inverse limit of the elements 
$$
 (U'(p),U'(p))^{-r} (e'_{\ord},e'_{\ord}) {}_{c}\mathcal{RI}(mp^{r},N,N') \in H^{1}_{S}(\Q(\mu_{mp^{r}}),H^{1}_{\ord}(Np^{\infty}) \widehat{\otimes} H^{1}_{\ord}(N'p^{\infty}) )
$$
with respect to the corestriction maps for all $r\geq 1$. 
Recall that $\omega : G_{\Q}\rightarrow \Z_{p}^{\times}$ denotes the cyclotomic character modulo $p$. For each residue class $i$ modulo $p-1$, we define
$$
{}_{c}\mathcal{BF}(N,N',\omega^{i}) \in H^{1}_{S}(\Q,H^{1}_{\ord}(Np^{\infty}) \widehat{\otimes} H^{1}_{\ord}(N'p^{\infty})\widehat{\otimes} \Lambda_{\Zp}^{\#}(\Gamma_{\cyc})(\omega^{i}))
$$
to be the corestriction of the twisted cohomology class 
$
{}_{c}\mathcal{BF}(1,N,N')^{(\omega^{i})}. 
$

\subsection{Families of CM modular forms}\label{FCMMFs}
We fix an algebraic closure $\overline{\Q}$ of $\Q$ and embeddings $\iota_{\infty}: \overline{\Q} \hookrightarrow \C$, $\iota_{p}: \overline{\Q} \hookrightarrow \C_{p}$, and we let $K$ be an imaginary quadratic field of discriminant $-D$. For an integral ideal $\mathfrak{n}$ of $K$, we let $K(\mathfrak{n})$ denote the maximal abelian $p$-extension of $K$ with conductor dividing $\mathfrak{n}$, and we let $H(\mathfrak{n})$ denote $\Gal(K(\mathfrak{n})/K)$. We define $\mathfrak{p}$ to be the prime of $K$ lying above $p$ which corresponds to $\iota_{p}$.

Suppose that $\psi$ is an algebraic Hecke character of $K$ with infinity type $(1,0)$ and modulus $\mathfrak{n}$. Then the $q$-expansion 
\begin{equation}\label{thetapsi}
\theta_{\psi}(q) = \sum_{(\mathfrak{a},\mathfrak{f})=1}\psi_{\infty}(\mathfrak{a})q^{N_{K/\Q}(\mathfrak{a})} \in \C[[q]]
\end{equation}
is the $q$-expansion of a normalised cuspidal eigenform of weight $2$, level $N:=N_{K/\Q}(\mathfrak{n})D$ and character $\varepsilon_{\psi}=\chi \varepsilon_{K}$. Here $\chi$ is the Dirichlet character defined by $\chi(a)=\psi((a))/a$, and $\varepsilon_{K}$ is the quadratic character associated to $K/\Q$. It is new precisely when $\mathfrak{n}$ is the conductor of $\psi$.  For details we refer the reader to  \cite[Theorem 4.8.2 ]{Miy}.

We now assume that $(\mathfrak{f}, p)=1$ is an integral ideal  of $K$  and  suppose that $\psi$ is an algebraic Hecke character of infinity type $(1,0)$ and conductor $\mathfrak{fp}^{r}$ for some $r\geq 0$. We then suppose that $(\mathfrak{n},p)=1$ is an integral ideal divisible by $\mathfrak{f}$ and put $N=N_{K/\Q}(\mathfrak{n})D$. We let $L$ be a finite extension of $\Qp$ with ring of integers $\mathcal{O}$, which is assumed to be large enough to contain the values of $\psi_{p}$. 
We denote by $\mathfrak{m}$ the maximal ideal of $\mathcal{O}$.

We have a ring homomorphism
\begin{align}\label{ringhom}
\phi_{\mathfrak{n}} : \mathbb{T}(Np^{\infty}) &\rightarrow \Lambda_{\mathcal{O}}(H(\mathfrak{np}^{\infty})) \\
 T'(\ell) &\mapsto \sum_{\mathfrak{l}} \psi(\mathfrak{l})\sigma_{\mathfrak{l}}^{-1}\nonumber \\
\left<d^{-1}\right> &\mapsto \varepsilon_{\psi}(d)\sigma_{(d)}^{-1} \nonumber.
\end{align}
where $\ell$ is prime and the sum runs over the ideals $\mathfrak{l}\nmid \mathfrak{np}$ of norm $\ell$. Here for an ideal $\mathfrak{a}$ of $K$ prime to $\mathfrak{np}$, $\sigma_{\mathfrak{a}}\in H(\mathfrak{np}^{\infty})$ is the image of $\mathfrak{a}$ under the Artin map, with the arithmetic normalisation.  To see that this defines a ring homomorphism, we note that for each finite order character $\chi : H(\mathfrak{np}^{\infty}) \rightarrow \overline{\Q}^{\times}$, $\chi_{p}\circ \phi_{\mathfrak{n}}$ coincides with the ring homomorphism on $\mathbb{T}_{\ord}(Np^{r})$, for $r$ sufficiently large, attached to the cuspform $\theta_{\chi^{-1}\psi}$.

\begin{definition} We define
$$
H^{1}_{\ord}(\psi,\mathfrak{np}^{\infty}):= H^{1}_{\ord}(Np^{\infty})\otimes_{\mathbb{T}(Np^{\infty}),\phi_{\mathfrak{n}}}\Lambda_{\mathcal{O}}(H(\mathfrak{np}^{\infty})).
$$

\end{definition}

In the following, let $S$ be a a finite set of primes containing those dividing $Np$.

\begin{definition}
For $\mathfrak{n}$ as above we define 
$$
{}_{c}\mathcal{BF}(\mathfrak{n},\psi,N',\omega^{i})\in H^{1}_{S}(\mathbb{Q},  H^{1}_{\ord}(\psi,\mathfrak{np}^{\infty}) \widehat{\otimes}  H^{1}_{\ord}(N'p^{\infty}) \widehat{\otimes} \Lambda^{\#}(\Gamma_{\cyc})(\omega^{i}))
$$
to be the image of ${}_{c}\mathcal{BF}(N,N',\omega^{i})$ under $\phi_{\mathfrak{n}}$, where $N = \N_{K/\Q}(\mathfrak{n}) D$. 
\end{definition}

\begin{theorem}\label{ind} Suppose that $\psi_{p\mid D_{\mathfrak{p}}} \not\equiv \psi_{p\mid D_{\mathfrak{p}}}^{c} \text{ mod }\mathfrak{m}$, where $D_{\mathfrak{p}}$ is a decomposition group at $\mathfrak{p}$. Then 
$H^{1}_{\ord}(\psi,\mathfrak{np}^{\infty})$ is free of rank $2$ over $\Lambda_{\mathcal{O}}(H(\mathfrak{np}^{\infty}))$. Moreover there exists an isomorphism between $H^{1}_{\ord}(\psi,\mathfrak{np}^{\infty})$ and  $\Ind_{K}^{\Q}\Lambda^{\#}_{\mathcal{O}}(H(\mathfrak{np}^{\infty}))(\psi_{p})$ as  $\Gal(\Q^{S}/\Q)$-representations over $\Lambda_{\mathcal{O}}(H(\mathfrak{np}^{\infty}))$. 
\end{theorem}

The following two results together with the control theorem allow us to deduce this theorem from the corresponding result at finite level. 

\begin{theorem}[Brauer--Nesbitt] Let $G$ be a finite group and let $F$ be a field. Suppose that $M_{1}$ and $M_{2}$ are finitely generated $F[G]$-modules with $\dim_{F}(M_{1})=\dim_{F}(M_{2})$, and $Tr_{M_{1}}(g) = Tr_{M_{2}}(g)$ for all $g\in G$. If $M_{1}$ is absolutely irreducible, then $M_{1}$ and $M_{2}$ are isomorphic as $F[G]$-modules. 
\end{theorem}
\begin{proof}
See Corollary 2.8 of \cite{Hida}.
\end{proof}

\begin{prop}[Carayol, Serre] Let $A$ be a pro-Artinian local ring with maximal ideal $\mathfrak{m}$ and finite residue field $k$. Suppose that $G$ is a profinite group and $M_{1},M_{2}$ are finitely generated free $A$-modules of the same dimension with $\mathfrak{m}$-adically  continuous $G$-actions. If $M_{1}/\mathfrak{m}M_{1}$ is an absolutely irreducible $k[G]$-module and $Tr_{M_{1}}(g)=Tr_{M_{2}}(g)$ for all $g\in G$, then $M_{1}$ and $M_{2}$ are isomorphic as $A[G]$-modules.
\end{prop}
\begin{proof}
See Proposition 2.13 of \cite{Hida}.
\end{proof}
We now prove Theorem \ref{ind}.
\begin{proof}
Let $\boldsymbol{m}$ denote the maximal ideal of $\mathbb{T}(Np^{\infty})$ containing $\ker(\phi_{\mathfrak{n}})$, and let $\mathbb{T}(Np^{\infty})_{\boldsymbol{m}}$ denote the localisation of the Hecke algebra at $\boldsymbol{m}$. By our assumption on $\psi_{p}$ we may apply \cite[Proposition 4.1.1]{EPW} to see that 
$$
H^{1}_{\ord}(Np^{\infty})_{\boldsymbol{m}} := H^{1}_{\ord}(Np^{\infty})\otimes_{\mathbb{T}(Np^{\infty})} \mathbb{T}(Np^{\infty})_{\boldsymbol{m}}
$$
is free of rank $2$ over $\mathbb{T}(Np^{\infty})_{\boldsymbol{m}}$. The first assertion now follows since $\phi_{\mathfrak{n}}$  extends to the localisation, and 
$$
H^{1}_{\ord}(\psi,\mathfrak{np}^{\infty}) = H^{1}_{\ord}(Np^{\infty})_{\boldsymbol{m}} \otimes_{\mathbb{T}(Np^{\infty})_{\boldsymbol{m}},\phi_{\mathfrak{n}}} \Lambda_{\mathcal{O}}(H(\mathfrak{np}^{\infty})).
$$

	We put $A=\Lambda_{\mathcal{O}}(H(\mathfrak{np}^{\infty}))$ and the modules $M_{1}=\Ind_{K}^{\Q} \Lambda_{\mathcal{O}}^{\#}(H(\mathfrak{np}^{\infty}))(\psi_{p})$ and $M_{2}= H^{1}_{\ord}(Np^{\infty}) \otimes _{\mathbb{T}(Np^{\infty}) , \phi_{\mathfrak{n}}}\Lambda_{\mathcal{O}}(H(\mathfrak{np}^{\infty}))$. Note that $M_{1}/\mathfrak{m}M_{1} \cong \Ind_{K}^{\Q} k(\overline{\psi}_{p})$ which is an absolutely irreducible $k[G_{\Q}]$-module by Mackey's irreducibility criterion and our assumption. Therefore, by the proposition above, it suffices to prove that $Tr_{M_{1}}(g)=Tr_{M_{2}}(g)$ for all $g\in G_{\Q}$. 

For a finite order character $\eta : H(\mathfrak{np}^{\infty}) \rightarrow  \C_{p}^{\times}$, let $L[\eta]$ denote the field generated over $L$ by the values of $\eta$ and let $\mathcal{O}[\eta]$ be its ring of integers. Then we define 
$$
\mathscr{P}_{\eta}=\ker\left(\eta : \Lambda_{\mathcal{O[\eta]}}(H(\mathfrak{np}^{\infty})) \rightarrow \mathcal{O}[\eta] \right).
$$
By the control theorem and multiplicity one we have
$$
M_{2} \otimes_{A} \left(\Lambda_{\mathcal{O}[\eta]}(H(\mathfrak{np}^{\infty}))/ \mathscr{P}_{\eta}\right) \cong V_{L[\eta]}(g_{\psi\eta^{-1}})^{*}.
$$
From the description of $V_{L[\eta]}(g_{\psi\eta^{-1}})^{*}$ in \eqref{charpolV}  for all primes $(\ell,Np)=1$
and all finite order characters $\eta$ as above we have $Tr_{M_{1}}(\Frob_{\ell}) \equiv Tr_{M_{2}}(\Frob_{\ell})$ modulo $\mathcal{P}_{\eta}$.  By the Cebotarev density theorem $Tr_{M_{1}}(g) \equiv Tr_{M_{2}}(g)$ modulo $\mathcal{P}_{\eta}$ for all $g\in G_{\Q}$. The result follows since the intersection of all such ideals is zero. 
\end{proof}

\subsection{Degeneracy maps}
  
Suppose that $\mathfrak{l}$ is a split prime of $K$ of norm $\ell$ with $(\ell,\mathfrak{n}p)=1$. We define the following homomorphism of $\Lambda_{\mathcal{O}}(H(\mathfrak{nlp}^{\infty}))[G_{\Q}]$-modules, 
$$
\mathcal{N}^{\mathfrak{nlp}^{\infty}}_{\mathfrak{np}^{\infty}} =  (\pr_{1})_{*} \otimes \pi  -   (\pr_{2})_{*}  \otimes \ell^{-1}\sigma_{\mathfrak{l}}^{-1}\psi(\mathfrak{l})\pi : H^{1}_{\ord}(\psi ,\mathfrak{nlp}^{\infty}) \rightarrow H^{1}_{\ord}(\psi, \mathfrak{np}^{\infty}),
$$
where $\pi : \Lambda_{\mathcal{O}}(H(\mathfrak{nlp}^{\infty})) \rightarrow  \Lambda_{\mathcal{O}}(H(\mathfrak{np}^{\infty}))$ is the restriction map. 

To see this gives a well defined homomorphism one uses the relations 
\begin{align*}
(pr_{i})_{*}\circ \left< d^{-1} \right> &=  \left< d^{-1} \right>\circ (pr_{i})_{*}  \text{ for } d\in Z_{N\ell} \text{ and } i=1,2\\
(pr_{i})_{*}\circ T'(q) &=  T'(q)\circ (pr_{i})_{*}  \text{ for }q\neq \ell \text{ prime and } i=1,2\\
(pr_{2})_{*} \circ T'(\ell) &= \ell (pr_{1})_{*}\\
(pr_{1})_{*} \circ T'(\ell) &= T'(\ell)\circ (pr_{1})_{*} - \left< \ell^{-1}\right> \circ (pr_{2})_{*}
\end{align*}
and checks that $\mathcal{N}^{\mathfrak{nlp}^{\infty}}_{\mathfrak{np}^{\infty}}$ respects the bilinearity relations between the elementary tensors. The following theorem is an elementary calculation using \eqref{twista}, \eqref{1stdeg} and \eqref{2nddeg}. 
\begin{theorem}\label{thm17}
Suppose that $\mathfrak{l}$ is a split prime of $K$ lying above the rational prime $\ell$ and assume that $(\ell,NN'p)=1$, where $N=\N_{K/\Q}(\mathfrak{n})D$. Then we have
$$
\mathcal{N}^{\mathfrak{nlp}^{\infty}}_{\mathfrak{np}^{\infty}} \left( {}_{c}\mathcal{BF}(\mathfrak{nl},\psi, N',\omega^{i}) \right) =P_{\mathfrak{l}}\left(\ell^{-1}\sigma_{\mathfrak{l}}^{-1}\otimes\sigma_{\ell}^{-1}\right){}_{c}\mathcal{BF}(\mathfrak{n},\psi, N',\omega^{i}),
$$
where $P_{\mathfrak{l}}(X) = 1-(1,T'(\ell))\psi(\mathfrak{l})\omega^{i}(\ell)X + (1,\left<\ell^{-1}\right>)\psi(\mathfrak{l})^{2}\omega^{2i}(\ell)X^{2}$. 
\end{theorem}
\begin{proof}
See \cite[Theorem 3.5.1]{LLZ2}.
\end{proof}

For $\mathfrak{l}$ as above, the map 
$
\mathcal{N}^{\mathfrak{nlp}^{\infty}}_{\mathfrak{np}^{\infty}}
$
factors through the projection
$$
 H^{1}_{\ord}(\psi, \mathfrak{nlp}^{\infty}) \rightarrow H^{1}_{\ord}(\psi, \mathfrak{nlp}^{\infty}) \otimes_{\Lambda_{\mathcal{O}}(H(\mathfrak{nlp}^{\infty})),\pi}    \Lambda_{\mathcal{O}}(H(\mathfrak{np}^{\infty})).
$$
We denote the resulting map by $\mathcal{N}^{\mathfrak{nlp}^{\infty}}_{\mathfrak{np}^{\infty}}$ also. 
\begin{prop}
Suppose that $\mathfrak{l}$ is a split prime of $K$ lying above the rational prime $\ell$ and assume that $(\ell,\mathfrak{n}p)=1$. 
Then the map 
$$
\mathcal{N}^{\mathfrak{nlp}^{\infty}}_{\mathfrak{np}^{\infty}}:  H^{1}_{\ord}(\psi,\mathfrak{nlp}^{\infty})  \otimes \Lambda_{\mathcal{O}}(H(\mathfrak{np}^{\infty})) \rightarrow H^{1}_{\ord}(\psi, \mathfrak{np}^{\infty}) 
$$
is an isomorphism of $\Lambda_{\mathcal{O}}(H(\mathfrak{np}^{\infty}))[G_{\Q}]$-modules.
\end{prop}
\begin{proof}
Note that both modules are finitely generated $\Lambda_{\mathcal{O}}(H(\mathfrak{np}^{\infty}))$-modules. By Nakayama's lemma it suffices to show that the induced map 
$$
\mathcal{N}^{\mathfrak{nlp}^{\infty}}_{\mathfrak{np}^{\infty}}:   H^{1}_{\ord}(\psi,\mathfrak{nlp}^{\infty}) \otimes \Lambda_{\mathcal{O}}(H(\mathfrak{np})) \rightarrow  H^{1}_{\ord}(\psi,\mathfrak{np}^{\infty})  \otimes  \Lambda_{\mathcal{O}}(H(\mathfrak{np}))
$$
is an isomorphism. However by the control theorem, this map can be identified with isomorphism 
$
\mathcal{N}^{\mathfrak{nlp}}_{\mathfrak{np}}
$
appearing in \cite[Proposition 5.2.5]{LLZ2}.  
\end{proof}
Now if $\mathfrak{n},\mathfrak{n}'$ are integral ideals of $K$ coprime to $p$ and such that $\mathfrak{n}$ divides $\mathfrak{n}'$ with $\mathfrak{n}^{-1}\mathfrak{n}'=\mathfrak{l}_{1}\cdots\mathfrak{l}_{k}$ a squarefree product of split primes of distinct norm which are prime to $\mathfrak{n}$, then we define 
$$
\mathcal{N}^{\mathfrak{n'p}^{\infty}}_{\mathfrak{np}^{\infty}} =   \mathcal{N}^{\mathfrak{nl}_{1}\mathfrak{p}^{\infty}}_{\mathfrak{np}^{\infty}} \circ \cdots \circ \mathcal{N}^{\mathfrak{nl}_{1}\cdots \mathfrak{l}_{k}\mathfrak{p}^{\infty}}_{\mathfrak{nl}_{1}\cdots \mathfrak{l}_{k-1}\mathfrak{p}^{\infty}}, 
$$
which can be shown to be independent of the ordering of $\mathfrak{l}_{1},\ldots,\mathfrak{l}_{k}$. 

For an integral ideal $\mathfrak{a}$ of $K$ we define the following subset of integral ideals,
$$\mathcal{R}(\mathfrak{a}) := \{\text{products }\mathfrak{r} \text{ of split primes with } (\mathfrak{r},\mathfrak{a})=1 \text{ and with squarefree norm} \}.$$  
\begin{corollary}
For all ideals $\mathfrak{r}\in \mathcal{R}(\mathfrak{f}p)$ we can find a family of isomorphisms
$$
\nu_{\mathfrak{r}} : H^{1}_{\ord}(\psi,\mathfrak{rfp}^{\infty}) \rightarrow \Ind_{K}^{\Q}\Lambda^{\#}_{\mathcal{O}}(H(\mathfrak{rfp}^{\infty}))(\psi_{p})
$$
such that for all ideals $\mathfrak{r},\mathfrak{r}^{\prime}\in \mathcal{R}(\mathfrak{f}p)$ with $\mathfrak{r}\mid \mathfrak{r}^{\prime}$  the following diagram commutes
$$
\begin{CD}\label{cd1}
H_{\ord}^{1}(\psi,\mathfrak{r}'\mathfrak{fp}^{\infty}) @>\nu_{\mathfrak{r}^{\prime}\mathfrak{f}}>> \Ind_{K}^{\Q}\Lambda^{\#}_{\mathcal{O}}(H(\mathfrak{r^{\prime}fp}^{\infty}))(\psi_{p})\\
@V\mathcal{N}^{\mathfrak{r^{\prime}fp}^{\infty}}_{\mathfrak{rfp}^{\infty}}VV @VVV\\
H^{1}_{\ord}(\psi,\mathfrak{rfp}^{\infty}) @>\nu_{\mathfrak{rf}}>> \Ind_{K}^{\Q}\Lambda^{\#}_{\mathcal{O}}(H(\mathfrak{rfp}^{\infty}))(\psi_{p})
\end{CD}
$$
\end{corollary}
\begin{proof} See \cite[Corollary 5.2.6]{LLZ2}. 
\end{proof}

\subsection{Definition of Euler system classes}

We now assume that $K'$ is an imaginary quadratic field of discriminant $-D'<0$ which is distinct from $K$ and such that $p$ is split in $K'/\Q$. We let $\mathfrak{p}'$ denote the prime of $K'$ above $p$ corresponding to $\iota_{p}$, and let $\overline{\mathfrak{p}}'$ be its conjugate. We further assume that $\psi'$ is an algebraic Hecke character of infinity type $(1,0)$ and conductor $\mathfrak{f}'\mathfrak{p}^{\prime s}$ where $(\mathfrak{f}',p)=1$ and $s\geq 0$. By extending scalars if necessary we may assume $\mathcal{O}$ is large enough so that it contains the values of $\psi'_{p}$. We also make the assumption  $\psi'_{p \mid D_{\mathfrak{p}'}} \neq \psi_{p \mid D_{\mathfrak{p}'}}'^{c}$ modulo $\mathfrak{m}$, where $D_{\mathfrak{p}'}$ is a decomposition group at $\mathfrak{p}'$. We let $N'=N_{K'/\Q}(\mathfrak{f}')D'$, and define the ring homomorphism  
$
\phi' : \mathbb{T}(N'p^{\infty}) \rightarrow \Lambda_{\mathcal{O}}(\Gamma ') 
$ as in \eqref{ringhom}. By the arguments above we may choose an isomorphism of $\Lambda_{\mathcal{O}}(\Gamma')[G_{\Q}]$-modules
\begin{equation}\label{iso'}
H^{1}_{\ord}(N'p^{\infty})\otimes _{\mathbb{T}(N'p^{\infty}),\phi'} \Lambda_{\mathcal{O}}(\Gamma') \cong \Ind^{\Q}_{K'}  \Lambda_{\mathcal{O}}^{\#}(\Gamma')(\psi_{p}'). 
\end{equation}

We choose a family of isomorphisms 
$$\nu_{\mathfrak{rf}} : H^{1}_{\ord}(\psi,\mathfrak{rfp}^{\infty}) \rightarrow \Ind_{K}^{\Q}\Lambda^{\#}_{\mathcal{O}}(H(\mathfrak{rfp}^{\infty}))(\psi_{p})
$$
for all $\mathfrak{r}\in \mathcal{R}(\mathfrak{f}cpN')$, which are compatible in the sense of Corollary \ref{cd1}. Applying $\nu_{\mathfrak{rf}}$, \eqref{iso'} and Shapiro's lemma we define 
${}_{c}\mathcal{BF}(\mathfrak{rf},\psi,\psi',\omega^{i})$ as the image of ${}_{c}\mathcal{BF}(\mathfrak{rf},\psi,N',\omega^{i})$ inside 
$H^{1}_{S}(F, \Lambda_{\mathcal{O}}^{\#}(\Gamma')(\psi_{p}')\widehat{\otimes} \Lambda^{\#}_{\mathcal{O}}(H(\mathfrak{rf}p^{\infty}))(\psi_{p}\omega^{i}))$, where $F=KK'$ denotes the imaginary biquadratic extension of $\Q$ generated by $K$ and $K'$, and $S$ is a finite set of places containing the infinite places and the primes dividing $NN'\N_{K/\Q}(\mathfrak{r})p$. Here we have made the identification of $H(\mathfrak{r}p^{\infty})$ with $H(\mathfrak{rp}^{\infty})\times \Gamma$ under the restriction maps.

We define $F_{\infty}$ to be the unique $\Z_{p}^{3}$-extension of $F$. Note that by considering ramification $F_{\infty}$ is the composite of the fields $\Q_{\cyc},K_{\infty}$ and $K'_{\infty}$. For an integral ideal $\mathfrak{r}$ of $K$, we define $F_{\infty}(\mathfrak{r})= F_{\infty}K(\mathfrak{r})$ and we write $G=\Gal(F_{\infty}/F)$ and $G(\mathfrak{r})= \Gal(F_{\infty}(\mathfrak{r})/F)$.   
For the following lemma we only need to assume that $p\neq 2$. 
\begin{lemma} Let $H$ and $H'$ denote the $p$-Hilbert class fields of $K$ and $K'$ respectively. Then $FH'/F$ and $FH/F$ are linearly disjoint. 
\end{lemma}
\begin{proof}
Let $A',A$ and $A_{F}$ denote the $p$-Sylow subgroup of the ideal class groups of $K',K$ and $F$  respectively. By class field theory we must show that the homomorphism $(N_{F/K'},N_{F/K}) : A_{F} \rightarrow A'\oplus A$ is surjective. However for any pair of fractional ideals $\mathfrak{a}'$ and $\mathfrak{a}$ of $K'$ and $K$ respectively, we have 
$$
(N_{F/K'},N_{F/K}) ((1/2)[\mathfrak{a}'\mathcal{O}_{F} \cdot \mathfrak{a}\mathcal{O}_{F}]) = ([\mathfrak{a}'],[\mathfrak{a}]). 
$$
Here we are using that $N_{F/K'}(\mathfrak{a}\mathcal{O}_{F}) = N_{K/\Q}(\mathfrak{a})\mathcal{O}_{K'}$ is principal, and similarly for $N_{F/K}(\mathfrak{a}'\mathcal{O}_{F})$.
\end{proof}
\begin{corollary}
For an integral ideal $\mathfrak{r}$ of $K$ prime to $p$, the natural map $G(\mathfrak{r}) \rightarrow \Gamma'\times H(\mathfrak{rp}^{\infty})\times \Gamma_{cyc}$ is an isomorphism.
\end{corollary}
\begin{proof} Note that $FK'_{\infty}K(\mathfrak{rp}^{\infty})/F$ and $F\Q_{cyc}/F$ are linearly disjoint, by considering ramification at the prime of $F$ above both $\overline{\mathfrak{p}}$ and $\overline{\mathfrak{p}}'$.
We now argue that $FK'_{\infty}/F$ and $FK(\mathfrak{rp}^{\infty})/F$ are linearly disjoint. If this were not the case, then there would exist a cyclic extension $F'/F$ of degree $p$ contained in both $FK'_{\infty}$ and $FK(\mathfrak{rp}^{\infty})$. Since $F'/F$ is a base change of subextensions of both $K'_{\infty}/K$ and $K(\mathfrak{rp}^{\infty})/K$, it follows that  $F'/F$ must be unramified, which contradicts the previous lemma. 
\end{proof}

\begin{definition} For an integral ideal $\mathfrak{r}\in \mathcal{R}(\mathfrak{f}cpN')$, we define
$$
{}_{c}\mathcal{BF}(\mathfrak{r},\psi,\psi',\omega^{i}) \in H^{1}_{S}(F, \Lambda_{\mathcal{O}}^{\#}(G(\mathfrak{r}))(\psi_{p}\psi'_{p}\omega^{i}))
$$
as the image of ${}_{c}\mathcal{BF}(\mathfrak{rf},\psi,N',\omega^{i})$ under \eqref{iso'} and corestriction. We also write ${}_{c}\mathcal{BF}(\psi,\psi',\omega^{i}) \in H^{1}_{S}(F, \Lambda_{\mathcal{O}}^{\#}(G)(\psi_{p}\psi'_{p}\omega^{i}))$ as the image of ${}_{c}\mathcal{BF}((1),\psi,\psi',\omega^{i})$ under corestriction.
\end{definition}

\subsubsection{Removing c}\label{removec}
Note that by \eqref{twista} and \eqref{cdsymm}, for integers $c,d$ coprime to $6NN'p$ and for an integral ideal $\mathfrak{r}\in \mathcal{R}(\mathfrak{f}pcdN')$ we have 
\begin{equation*}
\left[ d^2 - \varepsilon_{\psi}\varepsilon_{\psi'}\omega^{-2i}(d)\tau_{d} \right] {}_{c}\mathcal{BF}(\mathfrak{r},\psi ,\psi',\omega^{i}) = \left[  c^2 - \varepsilon_{\psi}\varepsilon_{\psi'}\omega^{-2i}(c)\tau_{c} \right] {}_{d}\mathcal{BF}(\mathfrak{r},\psi ,\psi',\omega^{i}),
\end{equation*}
where $\tau_{c}\in G(\mathfrak{r})$ is such that 
\begin{equation*}
\tau_{c\mid K(\mathfrak{rp}^{\infty})}=(c, K(\mathfrak{rp}^{\infty})/K)^{-1} \qquad \tau_{c\mid K_{\infty}'}=(c, K_{\infty}'/K')^{-1} 
\end{equation*}
$$
\tau_{c\mid \Q_{\cyc}}=(c,\Q_{\cyc}/\Q)^{2},
$$ 
and similarly for $\tau_{d}\in G(\mathfrak{r})$. We define 
\begin{equation}\label{deltac}
\delta_{c}= c^2 - \varepsilon_{\psi}\varepsilon_{\psi'}\omega^{-2i}(c)\tau_{c} \in \Lambda_{\mathcal{O}}(G(\mathfrak{r}))
\end{equation}
and $\delta_{d}$ similarly. 
If we assume the following condition 
\begin{equation}\label{hypc}
\varepsilon_{\psi}\varepsilon_{\psi'} \neq \omega^{2+2i} \text{ mod }\mathfrak{m}
\end{equation}
then may choose $d\geq 1$ coprime $6\N_{K/\Q}(\mathfrak{f})\N_{K'/\Q}(\mathfrak{f}')DD'p$  such that $\delta_{d}$ is invertible. Assuming \eqref{hypc}, for $\mathfrak{r}\in \mathcal{R}(\mathfrak{f}pdN')$ we define 
\begin{equation}\label{BFwoc}
\mathcal{BF}(\mathfrak{r},\psi ,\psi',\omega^{i}) = \delta_{d}^{-1} {}_{d}\mathcal{BF}(\mathfrak{r},\psi ,\psi',\omega^{i})
\end{equation}
where $d$ is chosen as above and is coprime to $\mathfrak{r}$. This is independent of choice of $d$, and can therefore be defined for all $\mathfrak{r}\in\mathcal{R}(\mathfrak{f}pN')$.

\subsection{Local properties}
In this subsection we state the local properties of the classes ${}_{c}\mathcal{BF}(\mathfrak{r},\psi,\psi',\omega^{i})$ at the primes of $F$ lying above $\overline{\mathfrak{p}}$. This extra information allows us to obtain stronger results when we apply the Euler system machinery of Rubin in Theorem \ref{ESM}. 
\begin{theorem}\label{lp} Let $\nu$ denote the prime of $F$ lying above both $\overline{\mathfrak{p}}$ and $\overline{\mathfrak{p}}'$. For all ideals $\mathfrak{r}\in \mathcal{R}(\mathfrak{f}cpN')$ we have  $\loc_{\nu}({}_{c}\mathcal{BF}(\mathfrak{r},\psi,\psi',\omega^{i})) = 0$.
\end{theorem}
\begin{proof}
We have an isomorphism
$$
 H^{1}(F_{\nu},\Lambda_{\mathcal{O}}^{\#}(G(\mathfrak{r}))(\psi_{p}\psi_{p}'\omega^{i})) \cong \varprojlim_{M} H^{1}(F_{\nu},\mathcal{O}^{\#}[\Gal(M/F)](\psi_{p}\psi_{p}'\omega^{i}))
$$
where $M$ runs over the finite extensions of $F$ contained in $F_{\infty}(\mathfrak{r})$. For each such $M$, let $\pi_{M}$ denote the projection map 
$$
\pi_{M}:  H^{1}(F_{\nu},\Lambda_{\mathcal{O}}^{\#}(G(\mathfrak{r}))(\psi_{p}\psi_{p}'\omega^{i})) \rightarrow H^{1}(F_{\nu},\mathcal{O}^{\#}[\Gal(M/F)](\psi_{p}\psi_{p}'\omega^{i})).
$$

It suffices to show that $\pi_{M}(\loc_{\nu}{}_{c}\mathcal{BF}(\mathfrak{r},\psi,\psi',\omega^{i}))=0$ for all $M$. 
For some integer $k$ sufficiently large, the class  $\pi_{M}({}_{c}\mathcal{BF}(\mathfrak{r},\psi,\psi',\omega^{i}))$ is the image of the motivic cohomology class ${}_{c}\Xi(p^{k}, \N_{K/\Q}(\mathfrak{fr})Dp^{k},N'p^{k})$ under the \'{e}tale regulator map and after twisting.  
It follows by a result of Nekovar and Niziol \cite[Theorem B]{NN} that  $\pi_{M}(\loc_{\nu}{}_{c}\mathcal{BF}(\mathfrak{r},\psi,\psi',\omega^{i}))$ lies in the subspace $H_{g}^{1}(F_{\nu},\mathcal{O}^{\#}[\Gal(M/F)](\psi_{p}\psi_{p}'\omega^{i}))$ as defined Bloch and Kato in \cite[(3.7.2)]{BK}.
However, under our assumptions on $\psi$ and $\psi'$, we have  
$$
H_{g}^{1}(F_{\nu},L^{\#}[\Gal(M/F)](\psi_{p}\psi_{p}'\omega^{i})) = 0.
$$
This can be seen from the dimension formula  
$$  
\dim H^{1}_{g}(G_{K}, V ) = \dim \D_{dR}(V )/\D^{+}_{dR}(V ) + \dim H^{0}
(G_{K}, V ) + \dim \D_{crys}(V^{*}(1))^{\phi =1}
$$
for a de Rham representation $V$.

By the long exact sequence of Galois cohomology, we see that  $\loc_{\nu}{}_{c}\mathcal{BF}(\mathfrak{r},\psi,\psi',\omega^{i})$ lies in the image of $\varprojlim H^{0}(F_{\nu}, \mathbb{D}^{\#}[\Gal(M/F)](\psi_{p}\psi_{p}'\omega^{i}))$, where $\mathbb{D}=L/\mathcal{O}$. We claim this module is zero. Indeed, by Shapiro's lemma 
$$
\varprojlim_{M} H^{0}(F_{\nu}, \mathbb{D}^{\#}[Gal(M/F)](\psi_{p}\psi_{p}'\omega^{i})) = \varprojlim_{M} \prod_{w \mid \nu} H^{0}(M_{w},W)
$$
where $W=\mathbb{D}(\psi_{p}\psi_{p}'\omega^{i})$ and $w$ runs over the primes of $M$ above $\nu$. This is zero when $i\not\equiv 0$ modulo $p-1$ by considering the action of inertia, so we may assume $i=0$. For each finite extension $M/F$, and place $w \mid \nu$, W is unramified at $w$
with $\Frob_{w}$ acting as multiplication by $\psi(\overline{\mathfrak{p}})^{f}\psi(\overline{\mathfrak{p}}')^{f}$ where $f=f(\omega/\nu)$ is 
the residue degree of $\omega$ over $\nu$. We let $M_{\omega,n}$ denote the degree $p^{n}$ extension in the 
cyclotomic $\Zp$-extension of $M_{\omega}$, and note that there exists $n_{0}\geq 0$ such that $M_{\omega,n}/M_{\omega,n_{0}}$ is totally ramified for all $n\geq n_{0}$. In particular  $W^{G_{M_{\omega,n}}}=W^{G_{M_{\omega,n_{0}}}}$ is a finite module and the transition map becomes multiplication by $p^{n-n_{0}}$. The result follows. 
\end{proof}

\subsubsection{Explicit reciprocity law}
The explicit reciprocity law relates the Euler system classes constructed above to a three variable $p$-adic $L$-function constructed by Hida. We first introduce the complex $L$-function whose values will be interpolated. For $\Theta$ an algebraic Hecke character of $F$ we define the complex $L$-function
$$
L(\Theta_{\infty},s) = \sum_{\mathfrak{a}} \Theta_{\infty}(\mathfrak{a}) \N(\mathfrak{a})^{-s} \qquad Re(s)>>0,
$$
where the sum runs over all integral ideals of $F$. We allow the possibility that $\Theta$ is imprimitive.

We let $I(\psi) \subseteq \mathrm{Frac}(\Lambda(\Gamma))$ denote the congruence ideal (see \cite[Section 7]{Hida15}) attached to the lambda-adic form $ \sum_{(\mathfrak{a},\mathfrak{fp})=1}\psi(\mathfrak{a})\sigma_{\mathfrak{a}}^{-1}q^{\N_{K/\Q}(\mathfrak{a})}\in \Lambda_{\mathcal{O}}(\Gamma)[[q]]$. We now describe Hida's three variable $p$-adic $L$-function attached to the lambda-adic CM forms through $\psi$ and $\psi'$. For each residue class $i$ modulo $p-1$, there exists a $p$-adic $L$-function $L_{p,i}^{\psi,\psi'} \in I(\psi)^{-1} \widehat{\otimes} \Lambda_{\mathcal{O}}(\Gamma') \widehat{\otimes} \Lambda_{\mathcal{O}}(\Gamma_{cyc})$ with the interpolation property given below. 
 
 Before stating the interpolation property it will be helpful to introduce some notation.  Recall that $\mathfrak{p}$ and $\mathfrak{p}'$ denote the primes of $K$ and $K'$ respectively corresponding to $\iota_{p}$. The primes $\mathfrak{p}_{1},\mathfrak{p}_{2},\mathfrak{p}_{3},\mathfrak{p}_{4}$ of $F$ lying above $p$ are determined by the splittings $\mathfrak{p}\mathcal{O}_{F} = \mathfrak{p}_{1}\mathfrak{p}_{2}$, and $\mathfrak{p'}\mathcal{O}_{F}= \mathfrak{p}_{1}\mathfrak{p}_{3}$. 
We suppose that $\rho$ and $\rho'$ are algebraic Hecke characters of $K$ and $K'$  of infinity types $(k,0)$ and $(k',0)$ respectively with $k,k'\geq 0$, such that the associated Galois representations $\rho_{p}$, $\rho'_{p}$ factor through $\Gamma$ and $\Gamma'$ respectively, and such that the conductors of $\psi'\rho'$ and $\psi'\rho'$ are prime to $p$. Assume also that $j\equiv -i$ modulo $p-1$  is an integer in the range 
$
k'< 1+j \leq k,
$
and let $\chi_{cyc}: \Gamma_{\cyc}\rightarrow 1+p\Zp$ denote the cyclotomic character restricted to $\Gamma_{cyc}$. 
To simplify notation we define $\chi:= \psi\rho$ and $\chi':=\psi'\rho'$, considered as Hecke characters with modulus $\mathfrak{f}$ and $\mathfrak{f}'$ respectively, and let $\Theta:=(\chi\circ \N_{F/K})\cdot (\chi'\circ \N_{F/K'})$.
Then we have 
$$
\int \left(\rho\rho'\chi_{\cyc}^{-j}\right)^{-1} d L_{p,i}^{\psi,\psi'}  = \frac{\mathcal{E}(\Theta,1+j)}{\mathcal{E}(\chi)\mathcal{E}^{*}(\chi)}\cdot \frac{j!(j-k')!L(\Theta_{\infty},1+j)}{\sqrt{-1}^{k'-k}\pi^{2+2j-k'} 2^{2+2j+k-k'} \left<\theta_{\chi},\theta_{\chi} \right>_{N}}.
$$
Here the Euler factors are defined by 
$$
\mathcal{E}(\chi) = \left( 1-\frac{\chi(\mathfrak{p})}{p\chi(\overline{\mathfrak{p}})} \right) \qquad \mathcal{E}^{*}(\chi) = \left( 1-\frac{\chi(\mathfrak{p})}{\chi(\overline{\mathfrak{p}})} \right),
$$
and
\begin{multline*}
\mathcal{E}(\Theta,1+j) = \left(1-\frac{\Theta(\mathfrak{p}_{1}) }{p^{j+1}}\right)\left(1-\frac{\Theta(\mathfrak{p}_{2})}{p^{j+1}}\right)\left(1-\frac{p^{j}}{\Theta(\mathfrak{p}_{3})}\right)\left(1-\frac{p^{j}}{\Theta(\mathfrak{p}_{4})}\right)\\ \times \prod_{\ell\mid (D,D')}\left(1- \frac{\chi(\mathfrak{l})\chi'(\mathfrak{l}')\varepsilon(\ell)}{\ell^{j+1}}\right).
\end{multline*}
In the final term  $\mathfrak{l}$ and $\mathfrak{l}'$ denote the primes above $\ell$ which ramifies in both $K$ and $K'$, and $\varepsilon$ is the quadratic character associated to the real quadratic subfield of $F$.  The cusp form $\theta_{\chi}$ is given by \eqref{thetapsi},  $N=\N_{K/\mathbb{Q}}(\mathfrak{f})D$  is its level and $\left<\theta_{\chi},\theta_{\chi} \right>_{N}$ is its Petersson norm.

We remark that this $p$-adic $L$-function is non-zero. Indeed it interpolates values which can be shown to be non-zero by appealing to the Euler product of the complex $L$-function in the region of absolute convergence. 

In the following theorem $H=G/\Gamma=\Gal(FK_{\infty}\Q_{cyc}/F)\cong \Z_{p}^{2}$. 
\begin{theorem}
Let $\nu$ denote the prime of $F$ lying above $\overline{\mathfrak{p}}$ and $\mathfrak{p}'$. We have a homomorphism of $\Lambda_{\mathcal{O}}(G)$-modules
$$
\mathcal{L}: H^{1}(F_{\nu}, \Lambda^{\#}_{\mathcal{O}}(G)(\vartheta\omega^{i})) \rightarrow I(\psi)^{-1}  \widehat{\otimes} \Lambda_{\mathcal{O}}(\Gamma') \widehat{\otimes} \Lambda_{\mathcal{O}}(\Gamma_{\cyc})
$$
such that the image of $\loc_{\nu}({}_{c}\mathcal{BF}(\psi,\psi',\omega^{i}))$ is given by $\delta_{c}L_{p,i}^{\psi, \psi'}$, where $\delta_{c}$ is given by \eqref{deltac}. Furthermore for each continuous character $\rho: \Gamma'\rightarrow \mathcal{O}^{\times}$, if we let ${}_{c}z(\psi,\psi',\omega^{i},\rho)\in  H^{1}(F_{\nu}, \Lambda^{\#}_{\mathcal{O}}(H)(\vartheta\rho\omega^{i})) $ be the image of ${}_{c}\mathcal{BF}(\psi,\psi',\omega^{i})^{(\rho)}$ under corestriction, then we have a $\Lambda_{\mathcal{O}}(H)$-module homomorphism 
$$
\mathcal{L}: H^{1}(F_{\nu}, \Lambda^{\#}_{\mathcal{O}}(H)(\vartheta\rho\omega^{i})) \rightarrow  I(\psi)^{-1} \widehat{\otimes} \Lambda_{\mathcal{O}}(\Gamma_{\cyc})
$$
such that $\mathcal{L}(\loc_{\nu} ({}_{c}z(\psi,\psi',\omega^{i},\rho))=\pi( \Tw_{\rho^{-1}}(\delta_{c}L_{p,i}^{\psi,\psi'}))$. Here $\pi$ is the projection $ I(\psi)^{-1}\widehat{\otimes} \Lambda_{\mathcal{O}}(\Gamma')   \widehat{\otimes} \Lambda_{\mathcal{O}}(\Gamma_{cyc})\rightarrow   I(\psi)^{-1} \widehat{\otimes} \Lambda_{\mathcal{O}}(\Gamma_{cyc})$.
\end{theorem}
\begin{proof} The first statement is the explicit reciprocity law of  \cite[Theorem 10.2.2]{KLZ2} and the second statement follows the same construction and the functoriality of the map in Theorem 8.2.3 \textit{op. cit.} 
\end{proof}

\section{Application to Selmer groups}\label{secsel}
We use the same notation as in the previous section and introduce the following notation. We let $K_{\infty,\infty}$ denote the unique $\Z_{p}^{2}$-extension of $K$, and let $N_{\infty}=FK_{\infty,\infty}$ and put $H=\Gal(K_{\infty,\infty}/K)\cong \Gal(N_{\infty}/F)$.

\subsection{Selmer groups}

Recall that $L$ be a finite extension of $\Qp$, $\mathcal{O}$ its ring of integers and $\mathfrak{m}$  is the maximal ideal of $\mathcal{O}$. We fix a continuous character $\theta :G_{F} \rightarrow \mathcal{O}^{\times}$, and we will use the notation  $T=\mathcal{O}(\theta)$, $V=T\otimes _{\mathcal{O}}L = L(\theta)$ and $W=V/T=\mathbb{D}(\theta)$.

We let $\Sigma(F)$ denote the places of $F$, and let $\Sigma_{p}(F)\subset \Sigma(F)$ denote the primes lying above $p$. We also suppose that $\Sigma_{p}(F)= \Sigma \cup \Sigma^{\prime}$ is a disjoint union of $\Sigma_{p}(F)$.   

For $v \in \Sigma(F)\backslash \Sigma_{p}(F)$, we define $H^{1}_{f}(F_{v}, V)=H^{1}_{\ur}(F_{v}, V) := \ker(H^{1}(F_{v},V) \rightarrow H^{1}(F^{\ur}_{v},V))$, where $F^{ur}$ is the maximal unramified extension of $F$. Furthermore we define $H^{1}_{f}(F_{v},T)$ and $H^{1}_{f}(F_{v},W)$ as the preimage and image of $H^{1}_{f}(F_{v}, V)$ under the natural maps arising from the short exact sequence 
$0\rightarrow T \rightarrow V \rightarrow W\rightarrow 0$. 
We then define $H^{1}_{s}(F_{v},W) = H^{1}(F_{v},W)/H^{1}_{f}(F_{v},W)$.

For a field $F \subset N \subset \overline{\Q}$ and $v\in \Sigma(F)$ we define 
$$
J_{v}(N) := \begin{cases}\varinjlim_{F \subset_{f} M \subset N} \bigoplus_{w\mid v} H^{1}_{s}(M_{w} , W) &\text{ if } v\notin \Sigma_{p}(F) \\
\varinjlim_{F \subset_{f} M \subset N} \bigoplus_{w\mid v} H^{1}(M_{\omega} , W)  &\text{ if } v\in \Sigma_{p}(F),
\end{cases}
$$
where the limit is taken over all finite extensions $M$ of $F$ contained inside $N$, and the sum over all the places of $M$ above $v$. Then the Selmer group of $W$ over $N$ with strict conditions at $\Sigma$ and relaxed conditions at $\Sigma'$ is defined by the exact sequence
$$
0\rightarrow \Sel^{\Sigma'}_{\Sigma }(N,W) \rightarrow H^{1}(N,W) \rightarrow \bigoplus_{v\in \Sigma(F)\backslash \Sigma'} J_{v}(N).
$$
The following two examples serve as motivation for our paper.  
\begin{example}\label{mainex} Suppose that $\theta : G_{F}\rightarrow \mathcal{O}^{\times}$ is a continuous finite order character of order prime to $p$ and suppose $\Sigma$ and $\Sigma'$ are as above. Let $F'=\mathbb{Q}^{\ker(\theta)}$, $\Delta=\Gal(F'/F)$ and define $M^{\Sigma'}_{\Sigma}(F')$ to be the maximal abelian $p$-extension of $F'$ which is unramified outside the primes above $\Sigma'$ and for which the primes above $\Sigma$ split completely. Note that $M^{\Sigma'}_{\Sigma}(F')$ has a natural $\Delta$ action given by conjugation. We also define $X^{\Sigma'}_{\Sigma}(F')=\Gal(M^{\Sigma'}_{\Sigma}(F')/F')$. Then restriction map 
$$
H^{1}(F,\mathbb{D}(\theta)) \rightarrow H^{1}(F',\mathbb{D}(\theta))^{\Delta} = \Hom(G_{F'},\mathbb{D})^{\theta^{-1}}
$$
is an isomorphism, since $p\nmid [F':F]$, and it identifies $\Sel^{\Sigma'}_{\Sigma}(F,\mathbb{D}(\theta))$ with $\Hom(X^{\Sigma'}_{\Sigma}(F'),\mathbb{D})^{\theta^{-1}}$. Applying $\Hom_{\mathcal{O}}(\cdot,\mathbb{D})$ we see that 
$$
\Hom_{\mathcal{O}}(\Sel^{\Sigma'}_{\Sigma}(F,\mathbb{D}(\theta)),\mathbb{D}) \cong \left(X^{\Sigma'}_{\Sigma}(F') \otimes_{\Zp} \mathcal{O} \right)^{\theta}. 
$$

Similarly we have an isomorphism of $\Lambda_{\mathcal{O}}(G)$-modules  
$$
\Hom_{\mathcal{O}}(\Sel^{\Sigma'}_{\Sigma}(F_{\infty},\mathbb{D}(\theta)),\mathbb{D}) \cong \left(X^{\Sigma'}_{\Sigma}(F'_{\infty}) \otimes_{\Zp} \mathcal{O} \right)^{\theta}. 
$$
where $F'_{\infty}$ is the composite of $F'$ and $F_{\infty}$ and $X^{\Sigma'}_{\Sigma}(F'_{\infty})$ is defined as above. 
\end{example}
\begin{remark}
For a prime $w$ of $F'$ we let $U_{1}(F'_{w})$ denote the units in the completion $\mathcal{O}_{F',w}$ which are congruent to $1$ modulo $w$, and we put $U_{1}(F'_{\Sigma'}) = \prod_{w\mid v\in \Sigma'} U_{1}(F'_{w})$. We also let $U(F')^{\Sigma'}_{\Sigma}$ denote the $\Sigma$-units of $F'$, which are congruent to $1$ modulo all the primes of $F'$ lying above $\Sigma'$. Then by class field theory and the finiteness of the class group, $\Sel^{\Sigma'}_{\Sigma}(F,\mathbb{D}(\theta))$ is finite if and only if the the homomorphism
$$
\left(U(F')^{\Sigma'}_{\Sigma} \otimes_{\Z} \mathcal{O}\right)^{\theta} \rightarrow \left( U_{1}(F'_{\Sigma'}) \otimes_{\Zp} \mathcal{O} \right)^{\theta}
$$
has finite cokernel. 
\end{remark}
 
\begin{example}
With the same notation as the previous example, we describe the Selmer group attached to 
$\mathcal{O}(\theta^{-1})(1)$. 

Let $Y(F')^{\Sigma'}_{\Sigma} \subset F'^{\times}\otimes \Qp/\Zp$ consist of the elements of the form $y\otimes p^{-n}$ with the local conditions
$
\ord_{w} (y) \equiv 0
$ modulo $p^{n}$ for all places $w$ not lying above $p$, and $y\in F_{w}^{\prime\times p^{n}}$ for all places $w$ lying above places in $\Sigma$. For an element $y\in F^{\prime\times}$ we define $(y)_{\Sigma'}$ to be the fractional ideal 
$$
(y)_{\Sigma'} = \prod _{\mathfrak{q}\nmid \mathfrak{p}\in \Sigma'} \mathfrak{q}^{\ord_{\mathfrak{q}}(y)}.
$$ 
Letting $A(F')$ denote the $p$-primary subgroup of the ideal class group of $F'$, we have a surjective homomorphism, $Y(F')^{\Sigma'}_{\Sigma} \rightarrow A(F')$ given by $(y\otimes 1/p^{n}) \mapsto [(y)_{\Sigma'}^{1/p^{n}}]$. We therefore have an exact sequence 
$$
0\rightarrow Z^{\Sigma'}_{\Sigma}(F') \rightarrow Y^{\Sigma'}_{\Sigma}(F') \rightarrow A(F')\rightarrow 0
$$
where 
$$
Z^{\Sigma'}_{\Sigma}(F') = \ker \left( \mathcal{O}_{F'}[1/\Sigma']^{\times}\otimes \Qp/\Zp \rightarrow \prod _{w\mid v\in \Sigma} \mathcal{O}_{F',w}^{\times} \otimes \Qp /\Zp \right)
$$
The restriction map together with the Kummer map give rise
to an isomorphism
$
H^{1}(F,\mathbb{D}(\theta^{-1})(1)) \cong \left(F'^{\times} \otimes \mathbb{D}\right)^{\theta}
$ which identifies $\Sel^{\Sigma'}_{\Sigma}(F,\mathbb{D}(\theta^{-1})(1))$ with $\left(Y^{\Sigma'}_{\Sigma}(F')\otimes_{\Zp} \mathcal{O}\right)^{\theta}$. It follows that  $\Sel^{\Sigma'}_{\Sigma}(F,\mathbb{D}(\theta^{-1})(1))$ is finite if and only if $\left(Z^{\Sigma'}_{\Sigma}(F')\otimes_{\Zp} \mathcal{O}\right)^{\theta}$ is finite. In particular, if any of the primes belonging to $\Sigma'$ split completely in $F'/F$, then $\Sel^{\Sigma'}_{\Sigma}(F,\mathbb{D}(\theta^{-1})(1))$ is of infinite order. 
\end{example}

We now compare the Selmer group defined over $N_{\infty}:=FK_{\infty,\infty}$ with the $\Gamma'$-invariants of the Selmer group defined over $F_{\infty}$. We thus consider the following commutative diagram with exact rows,
$$
\begin{CD}
0 @>>>\Sel^{\Sigma'}_{\Sigma}(N_{\infty},W) @>>> H^{1}(N_{\infty},W) @>>> \bigoplus_{v\in \Sigma(F)\backslash \Sigma'} J_{v}(N_{\infty}) \\
@. @V\alpha VV @V\beta VV @V\gamma VV \\
0 @>>>\Sel^{\Sigma'}_{\Sigma}(F_{\infty},W)^{\Gamma'} @>>> H^{1}(F_{\infty},W)^{\Gamma'} @>>> \bigoplus_{v\in \Sigma(F)\backslash \Sigma'} J_{v}(F_{\infty})^{\Gamma'},
\end{CD}
$$
where the vertical maps are the respective restriction maps.
\begin{lemma} Suppose that the residual character $\overline{\theta}: G_{F}\rightarrow \left(\mathcal{O}/\mathfrak{m}\right)^{\times}$ is not trivial. Then $\beta$ is an isomorphism. 
\end{lemma}
\begin{proof}
This is immediate from the inflation-restriction exact sequence. 
\end{proof}

\begin{lemma}
The map $\gamma$ is injective. 
\end{lemma}
\begin{proof}
For $v\in \Sigma(F)$, let $w$ be a place of $F_{\infty}$ above $v$, and let $w'$ denote the place of $N_{\infty}$ which lies below $w$. Then we define $F_{\infty,w} = \bigcup M_{w}$ as $M$ runs over the finite extensions of $F$ contained inside $F_{\infty}$, and   we define $N_{\infty,w}$ similarly.

In the case $v\nmid p$, $F_{\infty,w}$ is an unramified $p$-extension of $F_{v}$, which contains the cyclotomic $\mathbb{Z}_{p}$-extension of $F_{v}$. It follows that $F_{\infty,w}= F_{v,cyc}$ and similarly $N_{\infty,w}= F_{v,cyc}$.

In the case $v\mid p$ we claim that $N_{\infty,w}$ is equal to the maximal abelian $p$-extension of $F_{v} = \Qp$. By local class field theory this is a $\Z_{p}^{2}$-extension of $\Qp$ which is a composite of $\Q_{p,\cyc}$ and the maximal unramified $p$-extension of $\Qp$. 

Note that $N_{\infty}$ is the composite of $F$ and the unique $\Z_{p}^{2}$-extension of $K$. The claim now follows since the primes above $p$ in $K$ do not split infinitely in the $\Z_{p}^{2}$-extension of $K$. 
\end{proof}

\begin{corollary}
If the residual character $\overline{\theta}: G_{F} \rightarrow \left(\mathcal{O}/\mathfrak{m}\right)^{\times}$ is non-trivial then $\alpha$ is an isomorphism.
\end{corollary}

For $\Sigma, \Sigma'$ as above, we define 
$$
X^{\Sigma'}_{\Sigma}(F_{\infty},W) : = \Hom_{\mathcal{O}}( \Sel^{\Sigma'}_{\Sigma }(F_{\infty},W) ,\mathbb{D}),
$$
which is naturally a compact $\Lambda_{\mathcal{O}}(G)$-module. Similarly we define the  $\Lambda_{\mathcal{O}}(H)$-module $X^{\Sigma'}_{\Sigma}(N_{\infty},W)$. The results above show that if $\overline{\theta}$ is non-trivial, then we have an isomorphism of $\Lambda_{\mathcal{O}}(H)$-modules
$$
X^{\Sigma}_{\Sigma'}(F_{\infty},W)_{\Gamma'} \cong X^{\Sigma}_{\Sigma'}(N_{\infty},W).
$$

Under the isomorphism from Shapiro's we can identify $\Sel^{\Sigma'}_{\Sigma }(N_{\infty},W)$ with a certain subgroup of $H^{1}(K_{\infty,\infty}',\Ind^{K}_{F}W)$. From the discussion below we see that this subgroup can be interpreted as a Selmer group associated to the $G_{K}$-representation $\Ind^{K}_{F} W$. Given a place $v$ of $K$, we have an isomorphism of $G_{K_{v}}$-modules
$$
\Ind_{F}^{K} W \cong \prod_{w\mid v} \Ind_{F_{w}}^{K_{v}} W,
$$
where the product runs over the places $w$ of $F$ lying above $v$. Shapiro's Lemma then gives us a canonical isomorphism 
\begin{equation}\label{Shres}
H^{1}(K_{v},\Ind_{F}^{K} W) \cong \prod_{w\mid v} H^{1}(F_{w}, W).
\end{equation} 
The argument given in \cite[Section 3.1]{SU} shows that the above isomorphism identifies the unramified classes on each side. Furthermore, by \cite[Lemma I.3.5]{R}, $H^{1}_{f}(F_{w},W)$ is the maximal divisible subgroup of $H^{1}_{\ur}(F_{w},W)$, and similarly for $H^{1}_{f}(K_{v},\Ind^{K}_{F}W)$. It follows that \eqref{Shres} gives rise to an isomorphism 
$$
H^{1}_{f}(K_{v},\Ind_{F}^{K} W) \cong \prod_{w\mid v} H^{1}_{f}(F_{w}, W).
$$
Now suppose that $N/K$ is a finite field extension linearly disjoint from $F/K$. Then the 
following diagram commutes
$$
\begin{CD}
H^{1}(N,\Ind_{F}^{K}W) @>>> H^{1}(FN,W)\\
 @V \loc_{v} VV @VV \prod \loc_{w} V \\
H^{1}(N_{v},\Ind_{F}^{K}W) @>>>\prod_{w\mid v} H^{1}(FN_{w},W).
\end{CD}
$$
In particular for an arbitrary algebraic extension $N/K$ linearly disjoint from $F/K$, the isomorphism from Shapiro's lemma  identifies $\Sel^{\Sigma_{p}(F)}(FN,W)$ with $\Sel^{\Sigma_{p}(K)}(N,\Ind_{F}^{K}W)$.

When $\Sigma \cup \Sigma'$ is a disjoint union of $\Sigma_{p}(F)$, we define $\Sel^{\Sigma'}_{\Sigma}(N,\Ind_{F}^{K}W)$
 by the exact sequence
$$
0\rightarrow \Sel^{\Sigma'}_{\Sigma}(N,\Ind_{F}^{K}W) \rightarrow \Sel^{\Sigma_{p}(K)}(N,\Ind_{F}^{K}W) \rightarrow \bigoplus_{v\in \Sigma} J_{v}(NF),
$$
so that  $\Sel^{\Sigma'}_{\Sigma}(N,\Ind_{F}^{K}W)$ is identified with  $\Sel^{\Sigma'}_{\Sigma}(FN,W)$.

\subsection{Euler system and its twists}\label{esat}

In this subsection we describe the Euler system for the $G_{F}$-representation $\mathcal{O}(\vartheta)$, where $\vartheta:G_{F}\rightarrow \mathcal{O}^{\times}$ is a continuous character satisfying the following hypothesis. We recall that $\sigma$ and $\sigma'$ are generators of $\Gal(F/K)$ and $\Gal(F/K')$ respectively. 

\begin{hyp*} We assume that $\vartheta : G_{F}\rightarrow \mathcal{O}^{\times}$ can be written in the form $\vartheta = \chi_{p} \chi_{p}' \omega^{i}$ for some $i$ modulo $p-1$, and where $\chi$ and $\chi'$ are algebraic Hecke characters of $K$ and $K'$ of infinity types $(k+1,0)$ and $(k'+1,0)$ respectively with conductors coprime to $p$, such that the following conditions hold:  (i) $  \ker(\overline{\vartheta}_{G_{F(\mu_{p})}}^{\sigma}) \neq \ker(\overline{\vartheta}_{G_{F(\mu_{p})}})$, (ii)  $\overline{\vartheta}_{ \mid D_{\mathfrak{p}_{1}}} \neq \overline{\vartheta}_{ \mid D_{\mathfrak{p}_{1}}}^{\sigma}$  and (iii) $\overline{\vartheta}_{ \mid D_{\mathfrak{p}_{1}}} \neq \overline{\vartheta}_{ \mid D_{\mathfrak{p}_{1}}}^{\sigma'}$. 
\end{hyp*}

We first construct an Euler system for a twist of $\mathcal{O}(\vartheta)$ by a character of $G$. To describe this twist we need the following lemma. 

\begin{lemma} There exists an algebraic Hecke character $\eta$ of $K$ with modulus $\mathfrak{p}$, infinity type $(1,0)$, and such that associated Galois representation $\eta_{p}:G_{K}\rightarrow \overline{\Q}_{p}^{\times}$ factors through $\Gamma$.
\end{lemma}

\begin{proof} 
Let $I(\mathfrak{p})$ denote the group of fractional ideals of $K$ coprime to $\mathfrak{p}$. 
We first define $\eta_{p} : I(\mathfrak{p}) \rightarrow \overline{\Q}_{p}^{\times}$ which takes values in $\iota_{p}(\overline{\Q})$, and then put $\eta = \iota_{p}^{-1}\circ \eta_{p}$. We let $h$ denote the class number of $K$, and choose $\omega\in \overline{\Q}_{p}$ such that $\omega^{h} = 1+p$ and $\omega \equiv 1$ modulo $\mathfrak{P}$. Let $W=\omega^{\Zp}$, and let $r:1+p\Zp\rightarrow W$ be the unique $\Zp$-module homomorphism such that $r(1+p)=\omega$. We define $\eta_{p}$ as follows. Given an ideal $\mathfrak{a}\in I(\mathfrak{p})$, take $\alpha \in K$ such that $\alpha \mathcal{O}_{K} = \mathfrak{a}^{h}$, which is defined up to an element of $\mu_{K}$. Then we define 
$
 \eta_{p}(\mathfrak{a}) = r(\left<\iota_{p}(\alpha)\right>), 
$
where $\left<\cdot\right>$ denotes the projection map from $\Z_{p}^{\times}$ onto $1+p\Zp$. It is clear that the associated Galois representation factors through $\Gal(K(\mathfrak{p}^{\infty})/K)$. Since $\Gal(K(\mathfrak{p}^{\infty})/K_{\infty})$ is torsion and the image of $\eta_{p}$ is torsion free, the representation factors through $\Gamma$. 
\end{proof}

We choose algebraic Hecke characters $\eta,\eta'$ of $K,K'$ as in the lemma above. We then put $\psi= \eta^{-k}\chi$, $\psi'= \eta'^{-k'}\chi'$, so that $\psi$ and $\psi'$ have infinity types $(1,0)$ and conductors coprime to the prime below $\mathfrak{p}_{4}$. We let $\mathfrak{f}$ and $\mathfrak{f}'$ denote the prime to $p$ part of the conductors of $\psi$ and $\psi'$ respectively. As in the previous section, we define $N=\N_{F/K}(\mathfrak{f})D$, $N'=\N_{F/K'}(\mathfrak{f}')D'$, and choose an integer $(c,6NN'p)=1$. For  $\mathfrak{r}\in\mathcal{R}(\mathfrak{f}cpN')$, let $S$ denote a finite set of places containing the infinite places and those dividing $NN'\N_{K/\Q}(\mathfrak{r})p$.

For $\rho:\Gamma'\rightarrow \mathcal{O}^{\times}$ is a continuous character,
we define ${}_{c}z(\mathfrak{r},\vartheta\rho)$ as the image of ${}_{c}\mathcal{BF}(\mathfrak{r},\psi,\psi',\omega^{i})^{(\eta_{p}^{k}\eta_{p}^{\prime k'}\rho)}$ under the composition of maps
\begin{align*}
 H^{1}_{S}(F, \Lambda_{\mathcal{O}}^{\#}(G(\mathfrak{r}))(\vartheta\rho)) &\rightarrow  H^{1}_{S}(F, \Lambda_{\mathcal{O}}^{\#}(H(\mathfrak{r}p^{\infty}))(\vartheta\rho) )\\ 
&\cong H^{1}_{S}(K, \Lambda_{\mathcal{O}}^{\#}(H(\mathfrak{r}p^{\infty})) \otimes Ind^{K}_{F}\mathcal{O}(\vartheta\rho)),
\end{align*}
where the last isomorphism is given by Shapiro's lemma. We also write ${}_{c}z(\vartheta\rho)$ for the image of ${}_{c}z((1),\vartheta\rho)$ in  $H^{1}_{S}(K, \Lambda_{\mathcal{O}}^{\#}(H) \otimes Ind^{K}_{F}\mathcal{O}(\vartheta\rho))$, where we recall that $H=\Gal(K_{\infty,\infty}/K)\cong \Z_{p}^{2}$. The following theorem shows that the classes above define an Euler system for $(\Ind^{K}_{F}\mathcal{O}(\vartheta\rho),K_{\infty,\infty})$.
\begin{theorem}
Suppose that $\rho: \Gamma'\rightarrow \mathcal{O}^{\times}$ is a continuous character and suppose that $\mathfrak{rl}\in \mathcal{R}(\mathfrak{f}cpN')$, where $\mathfrak{l}$ is a split prime ideal of $K$. Then under the  restriction map $\pi:\Lambda_{\mathcal{O}}(G(\mathfrak{rl})) \rightarrow \Lambda_{\mathcal{O}}(G(\mathfrak{r}))$ we have 
$$
\pi({}_{c}z(\mathfrak{rl},\vartheta\rho)) = P_{\mathfrak{l},\vartheta\rho}(\Frob^{-1}_{\mathfrak{l}})\cdot {}_{c}z(\mathfrak{r},\vartheta\rho),
$$
where $P_{\mathfrak{l},\vartheta\rho}(X) = \det(1-X\cdot\Frob_{\mathfrak{l}}^{-1} \mid \Ind^{K}_{F}\mathcal{O}(\vartheta\rho)^{*}(1))\in \mathcal{O}[X]$.
\end{theorem}
\begin{proof} This follows from Theorem \ref{thm17} and \eqref{twista}.  
\end{proof}

We remark that by our assumption on $\vartheta$, the hypotheses $Hyp(K_{\infty,\infty},\Ind^{K}_{F}\mathcal{O}(\vartheta\rho))$ of \cite[II.3]{R} are satisfied by the Euler system above.
 
\begin{theorem}\label{ESM}
Suppose that $\rho : \Gamma' \rightarrow \mathcal{O}^{\times}$ is a continuous character, and put $W=\mathbb{D}(\vartheta\rho)$. Then we have 
\begin{multline*}
\Char_{\Lambda_{\mathcal{O}}(H)}\left( X^{\mathfrak{p}_{3},\mathfrak{p}_{4}}_{\mathfrak{p}_{1},\mathfrak{p}_{2}}(N_{\infty},W^{*}(1))\right) \text{ divides } \\
\Char_{\Lambda_{\mathcal{O}}(H)} \left( H^{1}(F_{\mathfrak{p}_{3}}, \Lambda_{\mathcal{O}}^{\#}(H)(\vartheta\rho)) / \Lambda_{\mathcal{O}}(H) \loc_{\mathfrak{p}_{3}}(z(\vartheta\rho)) \right).
\end{multline*}
\end{theorem}
\begin{remark} When $X$ is a finitely generated $\Lambda_{\mathcal{O}}(H)$-module which is not torsion, we adopt the convention that $\Char_{\Lambda_{\mathcal{O}}(H)}(X)=0.$
\end{remark}
\begin{proof}
By applying \cite[Theorem II.3.3]{R} we obtain the divisibility
$$
\Char_{\Lambda_{\mathcal{O}}(H)}(X^{\mathfrak{p}_{4}}_{\mathfrak{p}_{1},\mathfrak{p}_{2},\mathfrak{p}_{3}}(N_{\infty},W^{*}(1))) \text{ divides }\ind_{\Lambda_{\mathcal{O}}(H)}({}_{c}z(\vartheta\rho)),
$$
where the index $\ind_{\Lambda_{\mathcal{O}}(H)}(({}_{c}z(\vartheta\rho)))\subset \Lambda_{\mathcal{O}}(H)$ is as defined in \cite[II.3]{R}.
We note that the statement \textit{loc. cit.} concerns the Selmer group with the strict condition at all primes above $p$. However, using Theorem \ref{lp}, and the definition of the Kolyvagin derivative classes constructed in the proof \cite[Chapter IV]{R}, it follows that the derivative classes are locally trivial at $\mathfrak{p}_{4}$. The proof of \cite[Theorem II.3.3]{R} then strengthens to give the result above. 

The theorem now follows from applying the proof of \cite[Theorem II.3.8]{R} to the exact sequence 
$$
0 \rightarrow \frac{H^{1}(F_{\mathfrak{p}_{3}},\Lambda_{\mathcal{O}}^{\#}(H)(\vartheta\rho))}{\loc_{\mathfrak{p}_{3}} H^{1}(F,\Lambda_{\mathcal{O}}^{\#}(H)(\vartheta\rho))}   \rightarrow X^{\mathfrak{p}_{3},\mathfrak{p}_{4}}_{\mathfrak{p}_{1},\mathfrak{p}_{2}}(N_{\infty},W^{*}(1)) \rightarrow X^{\mathfrak{p}_{4}}_{\mathfrak{p}_{1},\mathfrak{p}_{2},\mathfrak{p}_{3}}(N_{\infty},W^{*}(1)) \rightarrow 0.
$$
\end{proof}

\section{Specialisations of Iwasawa modules}
In this section we let $\mathcal{O}$ be a finite extension of $\Zp$ and let $G$ be a group isomorphic to $\Z_{p}^{d}$ for some $d\geq 2$. We also let $\Gamma$ be a direct summand of $G$ which is isomorphic to $\Zp$ and put $H=G/\Gamma$. Given a finitely generated $\Lambda_{\mathcal{O}}(G)$-module $X$, and a continuous character $\rho:\Gamma \rightarrow \overline{\Q}_{p}^{\times}$, taking values in $\mathcal{O}'=\mathcal{O}[\rho]$ we can form the finitely generated $\Lambda_{\mathcal{O'}}(H)$-module $\Tw_{\rho}(X\otimes _{\mathcal{O}}\mathcal{O}')_{\Gamma}$. The main result of this section is Theorem \ref{charspec}, which shows that bounds for the characteristic ideal of $X$ can be recovered from bounds for the characteristic ideals of $\Tw_{\rho}(X)_{\Gamma}$, as $\rho$ varies over all continuous $\overline{\Q}_{p}$-valued continuous characters of $\Gamma$. The result is similar to Proposition 3.6 \cite{Och} and proved in the same way.

For notational purposes we present the results in this section in terms of modules over the power series ring  $\mathcal{O}[[T_{1},\ldots, T_{d}]]$. For any choice of topological generators $\gamma_{1},\ldots, \gamma_{d}$ for $G$, there is a unique continuous isomorphism $\varphi :\Lambda_{\mathcal{O}}(G) \rightarrow \mathcal{O}[[T_{1},\ldots, T_{d}]]$ such that $\varphi(\gamma_{i})=1+T_{i}$ for $i=1,\ldots,d$. We put $A=\mathcal{O}[[T_{1},...,T_{d-1}]]$ and $B=A[[t]]\cong \mathcal{O}[[T_{1},\ldots,T_{n}]]$. Given $\alpha \in 1+\mathfrak{m}$, where $\mathfrak{m}\subset \mathcal{O}$ is the maximal ideal,  we define the $A$-algebra homomorphism
\begin{align*}
\Tw_{\alpha} : B &\rightarrow B \\
f(t) &\mapsto f(\alpha(1+t)-1)
\end{align*}
for $f(t)\in A[[t]]=B$. 

Moreover for a $B$-module $X$ we define $\Tw_{\alpha}(X)$ to be the $B$-module with the same underlying abelian group as $X$, and multiplication defined by
$$
b\cdot x = \Tw_{\alpha}(b)x
$$
for $b\in B$ and $x\in X$.

We define $\pi : B \rightarrow A$ to be the natural map given by $\pi(f(t)) = f(0)$ for $f(t)\in A[[t]]=B$. We note that $\pi(\Tw_{\alpha}(f(t))) = f(\alpha -1)$. The following lemma describes the structure of $\Tw_{\alpha^{-1}}(X)$ in terms of that of $X$.

\begin{lemma}\label{twchar}
Suppose that $X$ is a finitely generated $B$-module, and $\alpha \in 1+\mathfrak{m}$. Then $X$ is a torsion $B$-module if and only if  $\Tw_{\alpha}(X)$ is a torsion $B$-module. Moreover we have
$$
\Char_{B}(\Tw_{\alpha^{-1}}(X))= \Tw_{\alpha}(\Char_{B}(X)). 
$$
\end{lemma}

\begin{lemma}\label{rank}
Suppose that X is a finitely generated $B$-module, and let $Y$ be its torsion submodule. Then we have 
$$
rank_{A}(X/tX) = rank_{B}(X) + rank_{A}(Y/tY).
$$
\end{lemma}
\begin{proof}
See the proof of \cite[Proposition 2]{G}. 
\end{proof}

The following proposition relates the characteristic ideal of $X$ as a $B$-module with the characteristic ideal of $X/tX$ as an $A$-module.  

\begin{prop}\label{spec} Suppose that $X$ is a finitely generated torsion $B$-module. Then we have the following equality of characteristic ideals
$$
ch_{A}(X_{t})\pi(ch_{B}(X)) = ch_{A}(X/tX),
$$
where $X_{t}:=\ker(X\xrightarrow{\times t} X)$. Moreover 
$$
ch_{A}(X_{t})=0 \iff \pi(ch_{B}(X)) =0 \iff ch_{A}(X/tX)=0.
$$
\end{prop}
\begin{proof}
See \cite[Proposition 2.10]{BBL}.
\end{proof}

It is necessary for us to extend scalars and consider twists of $X$ by characters valued in finite extensions of $\mathcal{O}$. Let $\mathcal{O}'$ be a finite extension of $\mathcal{O}$, and let $\mathfrak{m}'$ denote the maximal ideal of $\mathcal{O}'$. We define $A'=A\otimes_{\mathcal{O}}\mathcal{O}'$, $B'=B\otimes_{\mathcal{O}}\mathcal{O}'$, and for a $B$-module $X$ let $X'=X\otimes_{\mathcal{O}}\mathcal{O}'$. Then $B'= A'[[T]]$ and we let $\pi ' : B'\rightarrow A'$ denote the map $f(T)\mapsto f(0)$.  

\begin{theorem}\label{charspec} Let $X$ be a finitely generated $B$-module, and let $b\in B$. Suppose that for all finite extensions  $\mathcal{O}'$ of $\mathcal{O}$ and all $\alpha \in 1+\mathfrak{m}'$ we have
$$
\pi'(\Tw_{\alpha^{-1}}(b)) \in  \Char_{A'}(\Tw_{\alpha}(X') / t \Tw_{\alpha}(X')).
$$
Then $b\in \Char_{B}(X)$.
\end{theorem}
\begin{proof}
We may assume that $b\neq 0$, else there is nothing to prove. We first argue that if $b\neq 0$, then $X$ is a torsion $B$-module. By Lemma \ref{rank} it is sufficient to show that there exists $\alpha \in 1+\mathfrak{m}$ such that $\pi(\Tw_{\alpha^{-1}}(b))\neq 0$. We may write 
$$
b=\sum_{i_{1},\ldots, i_{d-1} \geq 0} g_{(i_{1},\ldots,i_{d-1})}(t)T_{1}^{i_{1}}\cdots T_{d-1}^{i_{d-1}}
$$
where $g_{(i_{1},\ldots,i_{d-1})}(t)\in \mathcal{O}[[t]]$. Note that for $\alpha \in 1+\mathfrak{m}$ we have 
$$
\pi(\Tw_{\alpha^{-1}}(b)) = \sum_{i_{1},\ldots, i_{d-1}} g_{(i_{1},\ldots,i_{d-1})}(\alpha^{-1} -1)T_{1}^{i_{1}}\cdots T_{d-1}^{i_{d-1}} \in A.
$$
By assumption there exists a $(d-1)$-tuple of non-negative integers $(i_{1},\ldots,i_{d-1})$ such that $ g_{(i_{1},\ldots,i_{d-1})}(t)\neq 0$. By the Weierstrass preparation theorem $g_{(i_{1},\ldots,i_{d-1})}(t)$ has only finitely many zeros in $\mathfrak{m}$. It follows that there exists $\alpha \in 1+\mathfrak{m}$ such that $\pi(\Tw_{\alpha^{-1}}(b))\neq 0$, and hence $X$ is a torsion $B$-module.  

	Suppose that $\Char_{B}(X) = (f(T_{1},\cdots,T_{d-1},t))$ and suppose $b=g(T_{1},\cdots,T_{d-1},t)$. By the assumption and Proposition \ref{spec}, we see that for all finite extensions  $\mathcal{O}'$ of $\mathcal{O}$ and all $\varepsilon \in \mathfrak{m}'$ we have
$$
f(T_{1},\ldots,T_{d-1}, \varepsilon) \text{ divides } g(T_{1},\ldots,T_{d-1}, \varepsilon).
$$
The result is then a consequence of the following lemma and its converse.
\end{proof}

\begin{lemma}
Let $n\geq 1$ and let $f,g \in \mathcal{O}[[T_{1},\ldots,T_{n}]]$. Suppose that for all finite extensions  $\mathcal{O}'$ of $\mathcal{O}$ and all $\alpha_{1},\ldots,\alpha_{n} \in \mathfrak{m}'$ we have $f(\alpha_{1},\ldots,\alpha_{n})$ divides $g(\alpha_{1},\ldots,\alpha_{n})$ in $\mathcal{O}'$. Then $f$ divides $g$ in $\mathcal{O}[[T_{1},\ldots,T_{n}]]$.
\end{lemma}
\begin{proof}
The lemma is proved by induction. Note that the statement is unchanged by passing to an extension of $\mathcal{O}$, making a linear change of variables, and by multiplying $f$ and $g$ by units. 

The case $n=1$ follows from the Weierstrass preparation theorem. Indeed suppose that $f(\alpha)$ divides $g(\alpha)$ for all $\alpha$ as above.  We may assume that $f$ is of the form
$$
f(T)=\pi^{m}(T-x_{1})^{e_{1}}\cdots (T-x_{k})^{e_{k}},
$$
where $\pi$ is a uniformiser of $\mathcal{O}$, $m,k\geq 0$, $x_{1},\ldots, x_{k}$ are distinct in  $\mathfrak{m} \subset \mathcal{O}$ and $e_{i}\geq 1$. Similarly we may suppose that
$$
g(T)=\pi^{n}(T-y_{1})^{f_{1}}\cdots (T-y_{l})^{f_{l}}
$$
with $n,l\geq 0$, $y_{1},\ldots, y_{l}$ distinct in $\mathfrak{m}$ and $f_{i} \geq 1$. 
Note that $v_{p}(f(\alpha)) \geq n$ for all $\alpha$, we $v_{p}$ is the discrete valuation on $\mathcal{O}$ normalised such that $v_{p}(\pi)=1$. However for all $0<\epsilon < \min_{i} v_{p}(y_{i})$ we can find $\alpha\in \overline{\Q}_{p}$  with $0<v_{p}(\alpha)<\epsilon$, and therefore we have 
$$
n v_{p}(\pi) \leq v_{p}(f(\alpha)) \leq v_{p}(g(\alpha)) = m v_{p}(\alpha) + \deg(g)\epsilon.
$$
Taking $\epsilon$ sufficiently small we conclude that $m\geq n$. By considering $\alpha = x_{i}$, we see that $l\geq k$ and after reordering, $x_{i}=y_{i}$ for $1\leq i \leq k$.  For $N>  \max_{2\leq i\leq l}v_{p}(x_{1}-y_{i})$ and for $\alpha = x_{1}+ p^{N}$ we have
$$
Ne_{1} \leq v_{p}(f(\alpha)) \leq v_{p}(g(\alpha)) \leq Nf_{1} + mv_{p}(\pi) + \deg(g) \max_{2\leq i\leq l}v_{p}(x_{1}-y_{i}). 
$$
By letting $N$ tend to infinity we see that $f_{1} \geq e_{1}$, and similarly $f_{i} \geq e_{i}$ for all $1\leq i \leq k$.

Now assume that $n\geq 2$. By extending scalars if necessary, making a linear change of variables, and multiplying by a unit we may assume that $f(T_{1},\ldots,T_{n}) = T_{n}^{r}+ b_{r-1}T_{n}^{r-1}+ \cdots + b_{0}$, where each $b_{i}$ belongs to the maximal ideal of $\mathcal{O}[[T_{1},\ldots, T_{n-1}]]$. For a proof see \cite[Lemma 3.8]{Och}.  The key remark is that $\mathcal{O}[[T_{1},\ldots,T_{n}]]/(f)$ is finite flat over $\mathcal{O}[[T_{1},\ldots,T_{n-1}]]$. 

For $j\geq 0$, let $I_{j}$ denote the ideal of $\mathcal{O}[[T_{1},\ldots,T_{n-1}]]$ generated by $T_{1}-jp$. Then for all $m\geq 0$, we have an injection
$$
\mathcal{O}[[T_{1},\ldots,T_{n-1}]]/I_{1}\cdots I_{m} \hookrightarrow \prod_{j=1}^{m} \mathcal{O}[[T_{1},\ldots,T_{n-1}]]/I_{j}.
$$
This remains injective after tensoring with $\mathcal{O}[[T_{1},\ldots,T_{n}]]/(f)$ over  $\mathcal{O}[[T_{1},\ldots,T_{n-1}]]$. In particular we have an injection
$$
\mathcal{O}[[T_{1},\ldots,T_{n}]]/(f,I_{1}\cdots I_{m}) \hookrightarrow \prod_{j=1}^{m} \mathcal{O}[[T_{1},\ldots,T_{n}]]/(f,I_{j})
$$
for all $m\geq 0$, and moreover 
$$
\mathcal{O}[[T_{1},\ldots,T_{n}]]/(f) \cong \varprojlim_{m\geq 0} \mathcal{O}[[T_{1},\ldots,T_{n}]]/(f,I_{1}\cdots I_{m}).
$$
Therefore $f$ divides $g$ if $g\in (f,I_{j})$ for all $j\geq 0$. Equivalently $f$ divides $g$ if $f(jp,T_{2},\ldots, T_{n})$ divides $g(jp,T_{2},\ldots, T_{n})$ in $\mathcal{O}[[T_{2},\ldots,T_{n}]]$ for all $j\geq 0$. We are now done by induction on $n$. 
\end{proof}

\section{Bounding Selmer groups}
We use the notation of sections \ref{seccon} and \ref{secsel}.
We may now prove the main theorem of this paper. 
\begin{theorem}\label{mainthm} Suppose that $\vartheta:G_{F}\rightarrow \mathcal{O}^{\times}$ is a continuous character satisfying $\Hyp(A)$ of section \ref{esat} and let $W=\mathbb{D}(\vartheta)$. Then  $X^{\mathfrak{p}_{2},\mathfrak{p}_{4}}_{\mathfrak{p}_{1},\mathfrak{p}_{3}}(F_{\infty},W^{*}(1))$ is a finitely generated torsion $\Lambda_{\mathcal{O}}(G)$-module and moreover
$$
\Char_{\Lambda_{\mathcal{O}}(G)} X^{\mathfrak{p}_{3},\mathfrak{p}_{4}}_{\mathfrak{p}_{1},\mathfrak{p}_{2}}(F_{\infty},W^{*}(1)) \text{ divides }   I(\psi)\delta_{c}L_{p,i}^{\psi,\psi'}.
$$
\end{theorem}

\begin{proof}
For all continuous characters $\rho : \Gamma' \rightarrow \mathcal{O}^{\times}$, we have 
\begin{multline*}
\Char_{\Lambda_{\mathcal{O}}(H)}\left( X^{\mathfrak{p}_{3},\mathfrak{p}_{4}}_{\mathfrak{p}_{1},\mathfrak{p}_{2}}(N_{\infty},W(\rho)^{*}(1))\right) \text{ divides } \\
\Char_{\Lambda_{\mathcal{O}}(H)} \left( H^{1}(F_{\mathfrak{p}_{3}}, \Lambda_{\mathcal{O}}^{\#}(H)(\vartheta\rho)) / \Lambda_{\mathcal{O}}(H) \loc_{\mathfrak{p}_{3}}({}_{c}z(\vartheta\rho)) \right)
\end{multline*}
by Theorem \ref{ESM}. Furthermore for any $\lambda\in I(\psi) \subset \Lambda_{\mathcal{O}}(\Gamma)$ we have we have a $\Lambda_{\mathcal{O}}(H)$-homomorphism 
\begin{equation}\label{erlrho}
\lambda \cdot \mathcal{L} : H^{1}(F_{\mathfrak{p}_{3}}, \Lambda_{\mathcal{O}}^{\#}(H)(\vartheta\rho)) \rightarrow \Lambda_{\mathcal{O}}(H)
\end{equation}
  such that the image of $\loc_{\mathfrak{p}_{3}}({}_{c}z(\vartheta\rho))$ is
  $\pi( \Tw_{\rho^{-1}}(\lambda\delta_{c}L_{p,i}^{\psi,\psi'}))$. Here $\pi$ is the projection $\Lambda_{\mathcal{O}}(G)\rightarrow  \Lambda_{\mathcal{O}}(H)$. By applying Tate's local duality formula, one can show that $H^{1}(F_{\mathfrak{p}_{3}}, \Lambda_{\mathcal{O}}^{\#}(H)(\vartheta\rho))$ is a rank one $\Lambda_{\mathcal{O}}(H)$-module. Furthermore $H^{1}(F_{\mathfrak{p}_{3}}, \Lambda^{\#}_{\mathcal{O}}(H)(\vartheta\rho))$ is torsion-free. Indeed it suffices to prove that for each $0\neq f\in \Lambda_{\mathcal{O}}(H)$, that $H^{0}(F_{\mathfrak{p}_{3}},(\Lambda^{\#}_{\mathcal{O}}(H)/ \Lambda^{\#}_{\mathcal{O}}(H)f) (\vartheta\rho))$ is trivial. Suppose that $\lambda\in \Lambda_{\mathcal{O}}(H)$ represents an element of $H^{0}(F_{\mathfrak{p}_{3}},(\Lambda^{\#}_{\mathcal{O}}(H)/ \Lambda^{\#}_{\mathcal{O}}(H)f) (\vartheta\rho))$. Then for all  we have 
\begin{equation}\label{torfree}
(\vartheta\rho(\sigma)\sigma_{H}^{-1} -1)\lambda \in \Lambda^{\#}_{\mathcal{O}}(H)f \text{ for all }\sigma \in G_{F_{\mathfrak{p}_{3}}},
\end{equation}
where $\sigma_{H}$ is the restriction of $\sigma$ to $H$. Note that $\mathfrak{p}_{3}$ splits finitely inside $N_{\infty}/F$, and so the decomposition group at $\mathfrak{p}_{3}$ inside $H$ is isomorphic to $\mathbb{Z}_{p}^{2}$. It follows that we can choose $\sigma,\sigma' \in  G_{F_{\mathfrak{p}_{3}}}$ such that $(\vartheta\rho(\sigma)\sigma_{H}^{-1} -1)$ and $(\vartheta\rho(\sigma')\sigma_{H}^{\prime -1} -1)$ are coprime. Since $\Lambda_{\mathcal{O}}(H)$ is a Krull domain, \eqref{torfree} implies that $\lambda\in \Lambda_{\mathcal{O}}(H)f$.

 It follows that \eqref{erlrho} is injective and  
\begin{multline*}\Char_{\Lambda_{\mathcal{O}}(H)} \left( H^{1}(F_{\mathfrak{p}_{3}}, \Lambda_{\mathcal{O}}^{\#}(H)(\vartheta\rho)) / \Lambda_{\mathcal{O}}(H) \loc_{\mathfrak{p}_{3}}({}_{c}z(\vartheta\rho)) \right) \Char_{\Lambda_{\mathcal{O}}(H)}(\coker(\lambda\cdot \mathcal{L}) \\ = \pi( \Tw_{\rho^{-1}}(\lambda\delta_{c}L_{p,i}^{\psi,\psi'})). 
\end{multline*}

Since we are free to extend scalars and choose $\rho:\Gamma' \rightarrow \overline{\Q}_{p}^{\times}$ to be an arbitrary continuous character, the theorem follows from Theorem \ref{charspec}.
\end{proof}

To obtain the result stated in the introduction, we apply the above theorem to the character $\vartheta = \theta^{-1}\omega$ to obtain a bound for the characteristic ideal of $X^{\mathfrak{p}_{3},\mathfrak{p}_{4}}_{\mathfrak{p}_{1},\mathfrak{p}_{2}}(F_{\infty},\mathbb{D}(\theta\omega^{-1})(1))$. Twisting  this $\Lambda_{\mathcal{O}}(G)$-module by a character of $\Gamma_{cyc}$, we obtain $X^{\mathfrak{p}_{3},\mathfrak{p}_{4}}_{\mathfrak{p}_{1},\mathfrak{p}_{2}}(F_{\infty},\mathbb{D}(\theta))$ which is described in Example \ref{mainex}, and we thus obtain a bound for its characteristic ideal by Lemma \ref{twchar}. 

Finally, we remark that if \eqref{hypc} is satisfied, then as in as in section \ref{removec}, we may remove the term $\delta_{c}$ from the statement of the above theorem.

\end{document}